\documentclass[leqno,12pt]{article}
\usepackage{lmodern}
\usepackage{enumerate}
\usepackage{amsmath}
\usepackage{cite}
\usepackage{amsfonts}
\usepackage{graphicx}
\usepackage{amsthm}
\usepackage{paralist}
\usepackage{mathrsfs}
\usepackage{pxfonts}
\usepackage{datetime}
\usepackage{comment}
\usepackage{fourier-orns}
\usepackage{dsfont}
\usepackage[cyr]{aeguill}
\usepackage{fancyhdr}
\usepackage{mathabx}
\usepackage[latin1]{inputenc}
\usepackage[all,cmtip]{xy}
\usepackage{multicol}
\usepackage{filecontents}
\usepackage{tikz-cd}
\usepackage{algorithmic}
\usepackage[ruled,vlined]{algorithm2e}
\usetikzlibrary{decorations.pathmorphing}
\usepackage[nottoc]{tocbibind}
\usepackage{verbatim}

\newcommand\restr[2]{{
\left.\kern-\nulldelimiterspace 
#1 
\vphantom{\big|} 
\right|_{#2} 
}}

\providecommand{\keywords}[1]
{
\small
\textbf{\textit{Keywords---}} #1
}

\newtheorem{theorem}{Theorem}[section]
\newtheorem{corollary}[theorem]{Corollary}
\newtheorem{lemma}[theorem]{Lemma}
\newtheorem{definition}[theorem]{Definition}
\newtheorem{remark}[theorem]{Remark}

\usepackage{geometry}
\newcommand{\sophie}[1]{{\color{red}\textsf{$\clubsuit\clubsuit\clubsuit$}  Sophie: [#1]}}
\newcommand{\mpendulo}[1]{{\color{blue}\textsf{$ \spadesuit \spadesuit \spadesuit$}  Mpendulo: [#1]}}

\title{The geometry of the moduli Space of non-cyclic biquadratic field extensions}
\author{MPENDULO CELE AND SOPHIE MARQUES}

\begin{document}
    \maketitle

    \begin{abstract}
        In this paper, we investigate the existence of an elementary abelian closure in characteristic not $2$ for biquadratic extensions. We discover that it exists for any non-cyclic extension. We make use of it to obtain a classification for this class of extensions up to isomorphism via descent. This permits us to describe the geometry of this moduli space in group theoretic terms. We also provide two families of polynomials with two parameters that can describe any quartic extensions. 
        
    \end{abstract}

\noindent \quad {\footnotesize MSC Code (primary): 11T22}

\noindent \quad {\footnotesize MSC Codes (secondary):  11R32, 11R16, 11T55, 11R58}


    \tableofcontents

    \section*{Introduction}\label{sec:introduction}
    
    In a previous work done by S. Marques and K. Ward in \cite{cubic}, \cite{MWcubic2}, \cite{MWcubic3}, the existence and the uniqueness of the purely cubic closure for cubics extension was identified and permitted to define purely cubic descent. These constructions have been used to classify cubic extensions as well as in computation of integral basis \cite{MWcubic4} and Riemann-Hurwitz formulae \cite{MWcubic5}. In this paper, we start a similar approach for biquadratic extensions. The exploration of the existence of the radical closure (see Definition \ref{radical}) for elementary abelian extensions revealed that radical extensions and elementary abelian extensions rarely intersect (see Theorem \ref{elem_radical}). Biquadratic extensions are quite different from cubic extensions. Fundamentally it begins with the two possible structure for groups of order $4$: $\mathbb{Z}/2\mathbb{Z} \times \mathbb{Z}/2\mathbb{Z}$ and $\mathbb{Z}/4\mathbb{Z}$ giving rise to two types of Galois extensions of degree $4$: the elementary abelian extensions and the cyclic extensions. Moreover, it was possible to nicely exhibit the geometry of the moduli space of the elementary abelian extension without the use of the radical closure (see Theorem \ref{final_elem_class}). Therefore, it became natural to wonder if in the quartic case, it could be that the notion of elementary abelian closure (see Definition \ref{eleclos}) becomes of particular interest and could be used to understand better the moduli space of those extensions up to isomorphism. In this paper we explore this idea. We are well aware that there is an extensive work that have been done on biquadratic extensions such as  \cite{Chu2002QuarticFA}, \cite{WQSQuartics}, \cite{Conrad1}. Nevertheless, our literature review did not find any, on the pursuit of the elementary abelian closure in order to be able to describe geometrically in group theoretic term the biquadratic extensions to understand its structure as a whole nor using the technics developed in this paper. The technics explored here are built so that they can be generalised for higher degree extensions in future. \\ \\
   
   In this paper, we start with setting the notations, definitions and basic results which will be used throughout the paper (\S 1). In the next section, we start with identifying two families of polynomials with two parameters that permit to construct any quartic extension (Lemma \ref{family} and Corollary \ref{familyc}). In doing so, we make explicit any change of variable that permits to pass from a general quadratic polynomial to one of those form. Among these two families, we find the biquadratic polynomials. The rest of the paper will focus on those. We first recall a few generalities about biquadratic extensions. In section \S 2.2, the reader will find a characterization of biquadratic extensions (see Lemma \ref{inter}), a criterium to determine the irreducibility of a biquadratic polynomial (see Lemma \ref{irred}), a criterium to determine the biquadratic extensions that are Galois (see Lemma \ref{galois_w}) and how to characterize the biquadratic extension according to their automorphism group (see Lemma \ref{galois_closure}). Those results are mostly well known in different form and context, for the sake of completeness we proved those for which we could not reference a proof in the given form. \\ \\ 
 
 The main goal of section \S 3 is to classify elementary abelian extensions so that we can describe the geometry of the moduli space of those extensions in group theoretic term. This is obtained in main theorem of this section: Theorem \ref{final_elem_class}. In doing so, we found various nice characterisations of elementary abelian extension (Theorem \ref{abelian-extension-properties}), that helped with the classification of elementary abelian extensions. Along the way, we identified when elementary abelian extension are radical, that happens precisely $F(i)$ is a quadratic subextension this extension (Theorem \ref{elem_radical}), which is a rare instance and explored a bit this path nevertheless. We discover that the radical closure may exist in other instances in the elementary abelian case but is non-unique, making it non applicable. \\ \\ 
 
Exploring a radical closure is not always of help for classification problems, in section \S 3.2, we identify when and how we can make use of the classification of elementary abelian extensions obtained in the previous section, using the notion of elementary abelian closure and descending along it. We realize in Lemma \ref{unique_elem_closure} that the existence of an elementary abelian closure is equivalent of being non-cyclic. In this case, the elementary abelian closure happens to be unique. The next result of the section, Theorem \ref{elem_closure_descent} permits to understand isomorphism classes of non-cyclic biquadratic extensions using the elementary abelian closure. In other words, it proves that isomorphism classes can be descended via the elementary abelian closure. In the remaining part of the paper, we highlight the geometry of the moduli space of those extensions. This conclude into the main result the paper Theorem \ref{final} that brings together all the pieces constructed previously into the same universe.

    \section{Background}\label{sec:background}
    \subsection{Basic context}
In the following, $F$ denotes a field of characteristic not $2$.
Given a quartic polynomial $P(X)$ we denote by $P'(X)$, $P''(X)$, $P'''(X)$ the first, second, third and forth derivative respectively.
If $\alpha \in L/F$ is an algebraic over $F$ we denote by $\min(\alpha, F)$ the minimal polynomial of $\alpha$ over $F$.
$\overline{F}$ denote an algebraic closure of $F$.

\begin{definition}
    Field extensions $L/F$ and $E/F$ are said to be \textbf{$F$-isomorphic} if there exist a ring isomorphism
    $\phi : L \rightarrow E$ such that $\phi(f)=f$ for all $f \in F$.
\end{definition}

\begin{definition}
    Let $K$ be a function field over $F$ and $L/K$ be a finite extension.
    The extension of function fields $L/K$ is said to be {\bf geometric} if $L \cap \overline{F}=F$.
\end{definition}

\begin{definition}
    Let $L/F$ and $K/F$ be field extensions, we assume that there is a field $E$ containing both $L$ and $K$.
    The \textbf{compositum} of $L.K$ is defined to be $L.K=F(L\cup K)$ where there right
    hand side denotes the extension of $F$  generated by $L$ and $K$ in $E$.
\end{definition}

We will make use of the following elementary lemma multiple times in this paper. We chose to include the proof for completeness. 
\begin{lemma}\label{quad}
    Let $L/F$ and $L'/F$ be quadratic field extension, $\gamma \in L$ and $\delta \in L'$ be primitive element such that $\gamma^2=a\in F$ and
    $\delta^2=b \in F$ then $L$ and $L'$ are $F$-isomorphic if and only if $a/b \in F^2$ or equivalently $\gamma/ \delta \in F$.
    Consequently, $$\{ L/F \text{ quadratic extension}\}/ \sim_{iso} \simeq \frac{ F^\times}{(F^{\times})^2} \backslash \{ (F^{\times})^2\} $$
    where $\sim_{iso}$ is the equivalence relation on the quadratic extensions $L/F\sim_{iso} L'/F$ if $L/F$ is $F$-isomorphic
    to $L'/F$.
\end{lemma}
\begin{proof}
    $\Longrightarrow$ Let $\psi : L \rightarrow L'$ be an $F$-isomorphism and $\gamma'=\psi(\gamma)$ then $(\gamma')^2=a$ and $\gamma'=c+d\delta$ for some $c,d \in F$.
    Replacing $\gamma'$ by $c+d\delta$ into the equation $\gamma'^2-a=0$ and using the uniqueness of the minimal polynomial, we prove that
    $c=0 \text{ and } a=d^{2}b$ as desired.\\
    $\Longleftarrow$ To show that $L/F$ and $K/F$ are $F$-isomorphic when $a/b \in F^2$, we write $a=d^{2}b$ for some $d \in F$ and prove that $d\delta$ is a primitive element of $F(\delta)$ with minimal polynomial $X^2-a$.

    The last isomorphism is obtained by sending the class of an extension up to isomorphism to the class of a radical parameter for some given radical generator up to squares.

\end{proof}

\subsection{Notation and terminology around quartic extensions}\label{subsec:notation}

In this paper, we will study and classify a special case of quartic extensions the non-cyclic biquadratic extensions.
\begin{definition} A field extension $L/F$ is said to be a \textbf{quartic extensions} if there exist a primitive element in $L$ that has a minimal polynomial of degree $4$.
\end{definition}

\begin{definition}
    A quartic field extension $L/F$ is said to be \textbf {biquadratic} if there exist a primitive element
    $\alpha \in L$
    such that $\min(\alpha, F) = X^4+uX^2+w$ where $u,w \in F$.
    Such $\alpha$ is called a biquadratic generator.
\end{definition}

Biquadratic Galois extensions are either cyclic or elementary abelian.
\begin{definition}
    A quartic field extension $L/F$ is said to be \textbf{cyclic} if $L/F$ is Galois and Galois group is isomorphic to
    $\mathbb{Z}/4\mathbb{Z}$.
\end{definition}

\begin{definition}
    A quartic field extension $L/F$ is said to be an \textbf{elementary abelian extension} if $L/F$ is Galois and
the Galois group is isomorphic to $\mathbb{Z}/2\mathbb{Z}\times\mathbb{Z}/2\mathbb{Z}$.
\end{definition}

\begin{remark}
    In the literature, one may also define a biquadratic field extension as a Galois extensions with
    Galois group isomorphic to $\mathbb{Z}/2\mathbb{Z}\times\mathbb{Z}/2\mathbb{Z}$.
    That definition is not equivalent to the definition we provided.
\end{remark}


\begin{definition}
    A finite field extension $L/F$ of degree $n$ is said to be \textbf{radical} if there exist a primitive element $\alpha \in 		L$ such that
    $\alpha^n \in F$.
    Such $\alpha$ is called a radical generator.
\end{definition}

\begin{definition} \label{radical} 
    Let $L/F$ be a field extension of degree $n$, a \textbf{radical closure} of $L/F$ (when it exists) is an extension $K/F$ of smallest
    degree such that $LK/K$ is a radical extension of degree $n$.
\end{definition}

\begin{remark}\label{remark_rmk}
\begin{enumerate}
    \item If $L/F$ be a biquadratic quartic extension admitting a radical closure $K$ then $L \cap K=F$.
    Indeed, pick a biquadratic generator $\alpha \in L$ with minimal polynomial $P(X)=X^4+uX^2+w$ over $F$.
    We argue by contradiction, suppose there exists $\theta \in L \cap K - F$, then $F(\theta)=L$ or $F(\theta)$ is a
    quadratic sub-extension of $L/F$, so $P(X)$ is reducible over $F(\theta)$ hence reducible over $K$.
    This proves that the minimal polynomial of $\alpha$ over $K$ is a proper divisor of $P(X)$ in $K[X]$ and $K(\alpha)/K$ is not a quartic field extension, contradicting the definition  of radical closure so $L \cap K = F$.
    \item In Theorem \ref{elem_radical} we prove that when an elementary abelian  extensions admits non-trivial radical
        closure then we have precisely $3$ non-isomorphic radical closures.
\end{enumerate}
\end{remark}

\begin{definition}\label{eleclos}
    Let $L/F$ be a field extension of degree $n$, an \textbf{elementary abelian closure} of $L/F$ (when it exists)
    is an extension $K/F$ of smallest degree such that $LK/K$ is an elementary abelian of degree $n$.
\end{definition}

\begin{definition}
    Let $L/K/F$ and $L'/K/F$ be a towers of fields, we say the {\bf towers are isomorphic} denoted as $L/K/F \cong L'/K'/F$ if there exist a ring
    isomorphism $\phi: L \rightarrow L'$ such that  $\restr{\phi}{K}$ is an $F$-isomorphism.
\end{definition}

    \section{Generalities about quartic extensions}\label{sec:biquadratic}
    
\subsection{Families of minimal polynomials with at most two parameters} 

In the following results, we are looking for a family of minimal polynomials with at most two parameters to represent all the quartic extensions. We found two of them the biquadratic polynomials and the polynomials of the form $T(X)=X^4+X^3+cX^2+d$ where $c$ and $d$ are in the ground field. 
\begin{lemma}\label{family}
    Given a quartic polynomial \[P(X) = X^4 + uX^3 + vX^2 + wX + z\]
    where $u,v,w,z \in F$ \\
    We denote $P_0= P\left(-\frac{u}{4}\right)$, $P_1= P'\left(-\frac{u}{4}\right)$, $P_2= P''\left(-\frac{u}{4}\right)$
    and $P_4=p''''(-\frac{u}{4})$
    \begin{enumerate}
        \item When $z=0$, then $P(X)$ is reducible with $0$ as a root;
        \item When $z\neq 0$, $P_0 = 0$, then $P(X)$ is reducible with $-\frac{u}{4}$ as a root;
        \item When $z\neq 0$, $P_0\neq 0$ and $P_1 = 0$, then $P(X)$ is irreducible if and only if $S(X)=P(X-\frac{u}{4})= X^4+ aX^2 + b$ is irreducible where $a =\frac{1}{2}P_2$ and $b= P_0$
        \item When $z\neq 0$, $u\neq 0$, $P_0\neq 0$ and $P_1\neq 0$, then the following statement are equivalent:
        \begin{enumerate} 
        \item $P(X)$ is irreducible 
        \item $R(X)= X^4 + aX^2+ bX+b$ is irreducible where $a= \frac{P_1^2 P_2}{2P_0^2}$ and $b= \frac{P_1^4}{P_0^3}$.
        \item $T(X)=X^4+X^3+cX^2+d$ is irreducible where $c= \frac{P_2 P_0}{2P_1^2}$ and $b= \frac{P_0^3}{P_1^4}$.
        \end{enumerate} 
    \end{enumerate}
\end{lemma}

\begin{proof}
    $1.$ and $2.$ are clear.\\
    For $3.$ and $4.$, we set $S(X)=P(X-u/4)$.
    Using tailor expansion, we have
    $$P(X)=P_0+P_1(X+\frac{u}{4})+\frac{1}{2}P_2(X+\frac{u}{4})^2+\frac{1}{6}P_3(X+\frac{u}{4})^3+\frac{1}{24}P_4(X+\frac{u}{4})^4$$
    since $P_3=P'''(-\frac{u}{4})=0$ and $P_4=P''''(X)=24$ we have
    $$P(X)=P_0+P_1(X+\frac{u}{4})+\frac{1}{2}P_2(X+\frac{u}{4})^2+(X+\frac{u}{4})^4$$

    hence
    $$S(X)=P(X-\frac{u}{4})=P_0+P_1X+\frac{1}{2}P_2X^2+X^4$$
    \begin{enumerate}
        \item [3.] When $P_1=0$ then $S(X)=X^4+\frac{1}{2}P_{2}X^2+P_0$, therefore the result is clear.
        \item[4.] Setting $R(X)=\frac{P_{1}^4}{P_{0}^4}S(\frac{P_0}{P_1}X)=
        \frac{P_1^4}{P_0^4}\bigg((\frac{P_0}{P_1}X)^4 + \frac{1}{2}P_2(\frac{P_0}{P_1}X)^2 + P_1(\frac{P_0}{P_1}X) + P_0 \bigg)=
        X^4 + \frac{P_1^{2}P_2}{2P_0^2}X^2 + \frac{P_1^4}{P_0^3}X + \frac{P_1^4}{P_0^3}$ and $T(X) =\frac{X^4}{b} R( \frac{1}{X})$, there the result is clear.
    \end{enumerate}

\end{proof}

\begin{corollary}\label{familyc}
    Let $L/F$ be a quartic field extension, $x \in L$ be a primitive element with minimal polynomial
     \[P(X) = X^4 + uX^3 + vX^2 + wX + z\]
    where $u,v,w,z \in F$ \\
    We denote $P_0= P\left(-\frac{u}{4}\right)$, $P_1= P'\left(-\frac{u}{4}\right)$ and $P_2= P''\left(-\frac{u}{4}\right)$.
    \begin{enumerate}
        \item When $P_1=0$ then $y=x+\frac{u}{4}$ is a primitive element with minimal polynomial\[R(X)=X^4+\frac{1}{2} P_{2}X^2+P_0\]
        \item When $P_1 \neq 0$ then \begin{enumerate} 
        \item  $y=\frac{P_0}{P_1}(x+\frac{u}{4})$ is a primitive element with minimal polynomial
        \[S(X)=X^4+aX^2+bX+b\] where $a= \frac{P_1^2 P_2}{2P_0^2}$ and $b= \frac{P_1^4}{P_0^3}$
        \item $z= \frac{4P_0}{P_1(4x+u)}$ is a primitive element with the minimal polynomial
            $T(X)=X^4+X^3+cX^2+d$  where $c= \frac{P_2 P_0}{2P_1^2}$ and $d= \frac{P_0^3}{P_1^4}$
        \end{enumerate} 
    \end{enumerate}
\end{corollary}
\subsection{Generalities about biquadratic extensions}\label{subsec:generalities-about-biquadratic-extensions}
In the rest of the paper, we will only focus on biquadratic extensions.

The following characterisation of biquadratic extension is well-known and very useful.

\begin{lemma}(\cite[Theorem 4.1]{Conrad2}) \label{inter}
    Let $L/F$ be a quartic field extension, then $L/F$ is a biquadratic if and only if $L/F$ has intermediate
    quadratic sub-extension.
\end{lemma}

The following lemma give a criterium to determine when a biquadratic polynomial is irreducible.

\begin{lemma} \label{irred}
    Let $F$ be a field and  $P(X) = X^4 + uX^2 + w \in F[X]$ with roots $\alpha,-\alpha,\beta,-\beta$ in it's splitting field,
    then the following are equivalent
    \begin{enumerate}
        \item $P(X)$ is irreducible over $F$
        \item $\alpha^2, \alpha+\beta $ and $\alpha-\beta$ are not in $F$
        \item $u^2-4w, -u+2\omega ,$ and $-u-2\omega$ are not $F^2$, where $\omega^2 =w$. 
    \end{enumerate}
\end{lemma}

\begin{proof}
    Let $\Delta=\alpha^2-\beta^2$ then $\Delta^2=u^2-4w$.
    Moreover, $P(X)$ is reducible if and only if $P(X)$ has a monic quadratic factor in $F[X]$.
    Indeed, the factor can be chosen to be monic since $F$ is a field.
    Moreover, if $P(X)$ is reducible over F and does not have a monic quadratic factor, the factor would have a root,
    say $\alpha$ but then $-\alpha$ is also a root and $X^2-\alpha^2$ divides $P(X)$. The converse is clear.
    We note also that over $\overline{F}$, we have $$ P(X) = (X - \alpha ) ( X+ \alpha ) (X- \beta) ( X+ \beta)$$
    Therefore, 
    there are $3$ ways to write $P(X)$ as a product quadratic polynomials, these are
    \begin{enumerate}
        \item[(a)] $\left( X^2-\alpha^2  \right) \left( X^2- \beta^2 \right)$.
        \item[(b)] $\big(X^2-(\alpha+\beta)X +\alpha\beta\big)\big(X^2+(\alpha+\beta)X +\alpha\beta\big)$.
        \item[(c)] $\big(X^2-(\alpha-\beta)X -\alpha\beta\big)\big(X^2+(\alpha-\beta)X -\alpha\beta\big)$.
    \end{enumerate}
    The factorisation $1.$ lies in $F[X]$ if and only if $\alpha^2 \in F^2$ if and only the discriminant $\Delta^2$ of
    $P(X^2)$ is a square in $F$.\\
    The factorisation $2.$ lies in $F[X] $ if and only if $\alpha + \beta \in F$.
    Indeed, $\alpha + \beta \in F$ implies that $\alpha \beta\in F$ since
    $ \alpha \beta = - \frac{1}{2} \left( \alpha ^2 + \beta^2 - (\alpha + \beta )^2 \right)$ and
    $\alpha^2 + \beta^2 = - u \in F$, this is true if and only if $(\alpha + \beta)^2 = \alpha^2 + \beta^2 +2 \alpha \beta= -u + 2 \omega \in F^2$ where $\omega^2 =w$.

    Finally, the factorisation $3.$ lies in $F[X]$ if and only if $\alpha - \beta \in F$.
    Indeed, $\alpha - \beta \in F$ implies that $\alpha \beta\in F$ since
    $ \alpha \beta = \frac{1}{2} \left( \alpha ^2 + \beta^2 - (\alpha - \beta )^2 \right)$ and
    $\alpha^2 + \beta^2 = - u \in F$, this is true if and only if $(\alpha - \beta)^2 = \alpha^2 + \beta^2 - 2 \alpha \beta= -u - 2 \omega \in F^2 $.
\end{proof}

\begin{remark}\label{inter_rmk}
	Let $F$ be a field and $P(X)=X^4+uX^2+w \in F[X]$. If $\alpha \in F$ is a root of $P(X)$ then $u^2-4w \in F^2$. This follows from the fact that $\alpha^2$ 
			is a root of $P(X^2)=X^2+uX+w$ in $F$. 
			In the case, when $P(X)$ is irreducible over $F$ and $L:=F[X]/{<}P(X){>}$ and $\alpha$ is a generator for $L/F$ with minimal polynomial $P(X)$, 
			we can deduce that $u^2-4w \in L^2$. We also have that $F(\alpha^2)= F(\gamma)$ is a quadratic sub-extension of $L/F$ with $\gamma^2 = u^2 - 4w$. 
\end{remark}

The following result characterise the Galois biquadratic extensions.

    \begin{lemma}\label{galois} \label{galois_w}
        Let $L/F$ be a biquadratic field extension, $\alpha \in L$ be a biquadratic generator with minimal polynomial
        $P(X)=X^4+uX^2+w$ over $F$.
        Then the following statements are equivalent,
        \begin{enumerate}
            \item $L/F$ is a Galois extension;
            \item $u^2-4w,u-2\omega$ and $u+2\omega$ are in $L^2$ with $\omega^2 = w$;
            \item $w \in L^2$.
        \end{enumerate}
    \end{lemma}

    \begin{proof}
    Under the present assumptions, $L/F$ is separable, therefore $L/F$ Galois is equivalent to $L$ is the
    splitting field of $P(X)$ over $F$.

        $1. \implies 2. \implies 3.$
        In the proof of the previous Lemma, we see that any of the factorisation $(a)$, $(b)$ and $(c)$ are all valid in
        the splitting field.
        Therefore $u^2-4w, -u+2\omega ,-u-2\omega  \in L^2$. In particular, we have $\omega  \in L$. So that $w \in L^2$

        $3. \implies 1.$  Let $\{\pm \alpha, \pm \beta\}$ be the roots $P(X)$ in the splitting field of $P(X)$, then
        $P(X)$ factors as $P(X)=(X^2-\alpha^2)(X^2-\beta^2)=X^4-(\alpha^2+\beta^2)X^2+(\alpha\beta)^2$.
        Now $w=(\alpha\beta)^2 \in L^2$ implies $\alpha\beta \in L$. So that, $\beta \in L$ since $\alpha \in L$ by assumption.  So $\{\pm \alpha, \pm \beta\}
        \subset L$ which proves $L$ is the splitting field of $P(X)$ over $F$. Hence, $L/F$ is Galois. 
    \end{proof}

    \begin{remark}\label{galois_rmk}
        Let $L/F$ be a biquadratic non-galois quartic extension, $\alpha \in L$ be a biquadratic generator with minimal polynomial
        $P(X)=X^4+uX^2+w$ over $F$. Then
        \begin{enumerate}
            \item $L^{Gal}=L(\omega)$ where $\omega \in \overline{F}$ and $\omega^2=w$, this is because the roots
                of $P(X)$ in $L^{Gal}$ are of the form $\pm \alpha, \pm \beta$ so $w=(\alpha\beta)^2 \in (L^{Gal})^2$,
                but $w \not\in L^2$ by the previous Lemma.
            \item  $Aut(L^{Gal}/F) \cong D_8$.
                This is because $[L^{Gal}:F]=[L^{Gal}:L][L:F]=8$,
                So $Aut(L^{Gal}/F)$ has $8$ elements, but the only transitive subgroup of $A_4$ of order $8$ is $D_8$.
                We know $Aut(L^{Gal}/F)$ is a transitive subgroup of $S_4$ by \cite[Theorem 2.9]{Conrad3}.
        \end{enumerate}
    \end{remark}

    The next result is well-known.  Since we could not find a reference that would prove the result in the present form, we decide to include the
    proof for completeness.
    This charactarization is useful in distinguishing between elementary abelian and non-elementary abelian extension
    in the classification of these extensions.
    \begin{lemma} \label{galois_closure}
        Let $L/F$ be a biquadratic field extension, $\alpha \in L$ be a primitive element with minimal polynomial
        $P=X^4+uX^2+w$, then $Aut(L/F)$ isomorphic to
        \begin{enumerate}
            \item  $\mathbb{Z}/2\mathbb{Z} \times \mathbb{Z}/2\mathbb{Z}$ if and only if $w \in F^2$
            \item  $\mathbb{Z}/4\mathbb{Z}$ if and only if $w \not\in F^2$ and $w(u^2-4w) \in F^2$
            \item  $\mathbb{Z}/2\mathbb{Z}$ if and only if $w \notin F^2$ and $w(u^2-4w) \notin F^2$
        \end{enumerate}
    \end{lemma}

    \begin{proof}
        \begin{enumerate}
            \item If $Aut(L/F) \cong \mathbb{Z}/2\mathbb{Z} \times \mathbb{Z}/2\mathbb{Z}$ then the discriminant of $P(X)$
                is a square in $F$ by \cite[Theorem 4.7]{Conrad3} since $\mathbb{Z}/2\mathbb{Z} \times \mathbb{Z}/2\mathbb{Z}$
                is a subgroup of $A_4$, that is $\Delta(P(X))=16w(u^2-4w)^2 \in F^2$ hence
                $w \in F^2$.
                Conversely, let $w = \omega^2$ for some $\omega \in F$.
                By Lemma \ref{galois} $L/F$ is Galois i.e
                $|Aut(L/F)|=[L:F]=4$, therefore $L/F$ is isomorphic to either
                $\mathbb{Z}/2\mathbb{Z} \times \mathbb{Z}/2\mathbb{Z}$ or $\mathbb{Z}/4\mathbb{Z}$.
                By Lemma \ref{galois} there exist $\gamma,\delta \in L$ such that
                $\gamma^2 = u^2-4w$ and $\delta^2=-u-2\omega$, in addition $F(\gamma)$ and $F(\delta)$ are quadratic
                sub-extensions of $L/F$.
                Those two quadratic extensions are distinct since otherwise $\frac{u^2-2w}{-u-2\omega}=-u+2\omega \in F^2$
                by Lemma \ref{quad} and this would contradict the irreducibility of $P(X)$ over $F$ by Lemma \ref{irred}.
                Since there is a one-to-one correspondence between quadratic sub-extensions of $L/F$ and subgroups
                $Aut(L/F)$ of degree $2$ it follows that $Aut(L/F)$ has at least two subgroups of order $2$ hence
                $Aut(L/F) \cong \mathbb{Z}/2\mathbb{Z} \times \mathbb{Z}/2\mathbb{Z}$.

            \item We suppose $Aut(L/F) \cong \mathbb{Z}/4\mathbb{Z}$ then $L/F$ is Galois and $w \not\in F^2$ by $1.$.
                $u^2-4w \not\in F^2$ by Lemma \ref{irred} and by Lemma \ref{galois} there exist $\gamma, \delta \in L$
                such that $\gamma^2=w$ and $\delta^2=u^2-4w$ so $F(\gamma)$ and $F(\delta)$ are quadratic
                sub-extensions of $L/F$.
                Since quadratic sub-extensions correspond with subgroups of $Gal(L/F)$ of degree $2$ and $Aut(L/F)$ has
                $1$ subgroup of order $2$ it follows that $F(\delta)=F(\gamma)$.
                As a consequence, $u^2-4w=a^{2}w$ for some $a \in F$ by Lemma \ref{quad}, this implies $w(u^2-4w)=(aw)^2 \in F^2$
                as desired.
                Conversely, let $w(u^2-4w) \in F^2$ and $w \not\in F^2$, we want to show that $w \in L^2$ which from
                Lemma \ref{galois} and $1.$ will prove the result.
                This follows from $u^2-4w \in L^2$ since it is the discriminant of $Q(X)=X^2+uX+w$ with $\alpha^2\in L$ as root. 
                Now that $w(u^2-4w),u^2-4w \in L^2$ it follows that $w \in L^2$.
            \item If $Aut(L/G) \cong \mathbb{Z}/2\mathbb{Z}$, then $w \not\in F^2$ by $1.$ and $w(u^2-4w) \not\in F^2$
                 by $2.$ Conversely if $u(u^2-4w) \not\in F^2$ and $w \not\in F^2$ then $L/F$ is not Galois by $1.$ and $2.$
                It then follows that $Aut(L/F)\cong \mathbb{Z}/2\mathbb{Z}$ or $Aut(L/F)$ is the trivial group.
                The automorphism $\sigma:\alpha \rightarrow -\alpha$ fixes $F$ and is not trivial hence
                $Aut(L/F)$ is not trivial. Therefore, $ Aut(L/F) \cong \mathbb{Z}/2\mathbb{Z}$

        \end{enumerate}

    \end{proof}

    \section{Elementary abelian fields}\label{sec:elementary-abelian-extension}
    
\subsection{Elementary abelian extensions}\label{subsec:elementary-abelain-extensions}
The following result characterize elementary abelian extensions via number of quadratic sub-extensions and description of the minimal polynomial.
\begin{theorem}\label{abelian-extension-properties}
Let $L/F$ be a quartic field extension, the following are equivalent
\begin{enumerate}
    \item $L/F$ is Galois and $Gal(L/F)=\mathbb{Z}/2\mathbb{Z} \times \mathbb{Z}/2\mathbb{Z}$.
    \item $L/F$ has exactly $3$ quadratic sub-extensions.
    \item $L/F$ has more than $1$ quadratic sub-extension.
    \item there exists $\gamma,\delta \in L$ such that $\gamma^2=a$ and $\delta^2=b$, $F(\gamma)$, $F(\delta)$ are quadratic extensions of $F$, $F(\gamma) \neq F(\delta)$ and $L = F(\gamma, \delta)$.
    \item $L/F$ has a biquadratic generator has with minimal polynomial of the $P(X)=X^4-2(a+b)X^2+(a-b)^2$ over $F$ for unique $a,b \in F$. 
    \item Any biquadratic generator of $L/F$ has a minimal polynomial has the form $$P(X)=X^4-2(a+b)X^2+ (a-b)^2$$ over $F$ for unique $a,b \in F$.
\end{enumerate}

Note that assuming $5.$ or $6.$ guarantee that there exists $\gamma,\delta \in L$ such that $\gamma^2=a$ and $\delta^2=b$, moreover
$F(\gamma)$, $F(\delta)$ are quadratic extensions of $F$, $F(\gamma) \neq F(\delta)$, that is $a/b\notin F^2$ and $L = F(\gamma, \delta)$.
Also $\alpha+\beta$ is a primitive element of $L/F$ with minimal polynomial $P(X)$.
\end{theorem}

\begin{proof} We will prove that
    $$1. \implies 2. \implies 3. \implies 4. \implies 5. \implies 1. \implies 6.\implies 1.$$

    $1. \implies 2.$ $L/F$ is Galois hence there is a one-to-one correspondence between quadratic sub-extensions of
    $L/F$ and subgroups of $\mathbb{Z}/2\mathbb{Z} \times \mathbb{Z}/2\mathbb{Z}$ of order $2$, so $L/F$ has $3$
    quadratic sub-extensions.\\
    $2. \implies 3.$ is trivial\\
    $3. \implies 4$ Let $F(\gamma)$ and $F(\delta)$ be distinct quadratic subextensions of $L/F$ such that
    $\gamma^2 \in F$ and $\delta^2 \in F$. 
   Then $[F(\gamma,\delta):F]
   =4$ concluding the proof.\\
    $4. \implies 5$ Let $F(\gamma)$ and $F(\delta)$ be distinct quadratic sub-extensions of $L/F$, We will prove $P(X)=X^4-2(a+b)X^2+(a-b)^2$
    is the minimal polynomial of $\gamma+\delta$.
    Indeed, firstly, by completing the square, we obtain $P(\gamma + \delta )=((\gamma + \delta)^2 - ( \gamma^2 + \delta^2))^2 - (\gamma^2 + \delta^2)^2 +(\gamma^2-\delta^2)^2=0$ and we will prove that $P(X)$ is irreducible over $F$.
    We have similarly that $\gamma-\delta$ is also a root in $P(X)$ in $L$.
    Therefore, the roots of $P(X)$ are $\pm \alpha$, $\pm \beta$ where $\alpha = \gamma + \delta$ and $\beta = \gamma-\delta$.
    But $\alpha^2 = a+b +2 \gamma \delta \not\in F$ because 
    $\frac{\gamma}{\delta} \notin F$. 
    Moreover, neither $ \alpha+ \beta= 2\gamma$, nor $\alpha -\beta = 2 \delta$ are in $F$,
    since $F(\gamma)$ and $F(\delta)$ are quadratic extensions of $F$.\\
    $5. \implies 1.$ Follows from Lemma \ref{galois_closure} \\
    $1. \implies 6.$ We assume that $L/F$ is Galois with $Gal(L/F)=\mathbb{Z}/2\mathbb{Z} \times \mathbb{Z}/2\mathbb{Z}$.
    Let $\alpha \in L$ be a biquadratic generator and $P(X)$ be the minimal polynomial of $\alpha$ over $F$,
    then by Lemma \ref{galois_closure}, $P(X)=X^4+uX^2+w$ for some $u,w \in F$ with $w= \omega^2$ for some $\omega \in F$.
    Let  $a=\frac{1}{4}(2\omega-u)$ and $b=-\frac{1}{4}(2\omega+u)$  then $P(X)=X^4-2(a+b)X^2+ (a-b)^2$.
    Note that $a,b \not\in F^2$ by Lemma \ref{irred} and $a,b \in L^2$ since $L/F$ is a splitting field of $P(X)$,
    that is there exists $\gamma,\delta \in L$ such that $\gamma^2=a$ and $\delta^2=  b$.
    By lemma \ref{quad}  $F(\gamma)=F(\delta)$ if and only if $a,b \in F^2$, hence to show that $F(\gamma) \neq F(\beta)$
    it suffices to show that $a/b \not\in F^2$.
    On the contrary suppose $a=k^{2}b$ where $k \in F$, so $u=-2(a+b)=-2b(k^2+1)$ and $w=(a-b)^2= b^2(k^2-1)$ then
    $u^2-4w=(4bk)^2 \in F^2$ contradicting that $P(X)$ is irreducible over $F$, so $F(\gamma) \neq F(\delta)$
    as required.\\
    $6. \implies 1.$ Follows from Lemma \ref{galois_closure}

\end{proof}
\begin{remark}
Note that, if $F$ is a field and $P(X)=X^4-2(a+b)X^2+(a-b)^2 \in F[X]$ as in the previous theorem then $P(X)$ is irreducible over $F$ if and only
    $a \not\in F^2$ and $b \not\in F^2$ and $ab \not\in F^2$ by Lemma \ref{irred}. Moreover, the roots of this polynomial are 
    $\epsilon_1 \gamma  + \epsilon_2 \delta$, where $\alpha$ (resp. $\gamma$) is such that $\gamma^2 =a$ (resp. $\delta^2 = b$) and $\epsilon_i \in \{ \pm 1 \}$, $i =1,2$. 
    Also, the possible quadratic factorisation of of $P(X)$ are now  
    $\left( X^2-a - b - 2 \gamma \delta  \right) \left( X^2- a -b +2 \gamma \delta \right) $, 
    $\big(X^2-2\gamma X +a-b \big)\big(X^2+2 \gamma X +a -b \big) $
    and $ \big(X^2-2 \delta X - a+b \big)\big(X^2+ 2 \delta X -a +b \big)$.
\end{remark}

        A nice application of the previous theorem, describe radical elementary abelian extensions.  Part of the following result is also part of \cite[Theorem 6.1.]{WQSQuartics}. We start with a remark we will use multiple times in the proof.

    \begin{remark}\label{rad_min_poly_remains_rmk}
        Let $L/F$ be an elementary abelian extension, if $K$ is a radical closure of $L/F$ then $L.K/K$ is an elementary
        abelian extension, if $\alpha$ is a biquadratic generator of $L/F$ then $L.K/K=K(\alpha)$ and
        $min(\alpha,F)=min(\alpha,K)$ where the second equality follows because $min(\alpha,F),min(\alpha,K)$ are both
        degree $4$ monic elements of $K[X]$ having $\alpha$ as a root.
    \end{remark}

    \begin{theorem}\label{elem_radical}
        Let $L/F$ be a elementary abelian extension then $L/F$ has
        \begin{enumerate}
            \item a trivial radical closure if and only if  $F(i)$ is a quadratic sub-extension of $L/F$. In this case $L/F$ is radical, and any radical generator is of the form $\gamma(i + 1)$ where $\gamma \in L$ is a radical generator of a quadratic sub-extension of $L/F$ and $F(\gamma) \neq F(i)$. 
            \item no radical closure if and only if $i \in F$;
            \item exactly $3$ radical closures if and only if $i \not\in L$.
              In this case, the $3$ radical closure are precisely $F(\gamma i)$ where $\gamma$ is a radical generator of a quadratic subextension of  $L/F$.
        \end{enumerate}
    \end{theorem}

    \begin{proof}
        \begin{enumerate}
            \item Let $L/F$ be a radical extension and $\alpha \in L$ be a radical generator of $L/F$, then
                 $min(\alpha,F)=P(X)=X^4+w$ where $w=\omega^2$ for some $\omega \in F$.
                Moreover, $P(X)$ is irreducible over $F$ hence $-4w=-4\omega^2 \not\in F^2$ by Lemma \ref{irred} this implies
                $-1 \not\in F^2$ hence $i \not\in F$.
                By Remark \ref{inter_rmk}, $-4w=-4\omega^2 \in L^2$ hence $i \in L$ so indeed $F(i)$ is a quadratic
                sub-extension of $L/F$.

                Conversely, let $F(i)$ be a quadratic sub-extension of $L/F$. Pick $F(\gamma)$ a quadratic sub-extension
                of $L/F$ different from $F(i)$ such that $\delta^2=a \in F$, then
                $-a \not\in F^2$ by Lemma \ref{quad}.
                We claim that $\delta(1+i) \in L$ is a radical generator for $L/F$ with minimal polynomial $Q(X)=X^4+4a^2$.
                First, $\delta(1+i) \in L$ is a root of $Q(X)$ since $Q(\delta(1+i))=(\delta(1+i))^4+(2a)^2=0$.
                Also, we prove that $Q(X)$ is irreducible. For this we apply Lemma \ref{irred} for $Q(X)$ so that $u=0$ and
                $w= 4a^2$ and therefore $u^2-4w=-(4a)^2 \not\in F^2$ as
                $-1 \not\in F^2$, $-u+2\omega =4a \not\in F^2$ as $a \not\in F^2$ , lastly
                $-u-2\omega =-4a \not\in F^2$ as $-a \not\in F^2$, proving the irreducibility of $Q(X)$ and concluding the proof 					    of the claim.

                We now show that any radical generator 
                is of the form $\gamma(i + 1)$ for some
                $\gamma$ is a radical generator of a quadratic sub-extension of $L/F$.
                Let $min(\beta, F)=X^4+w^2$. We have $w=a-b$ and $0=a+b$ for some $a,b \in F \setminus F^2$. 
            		So that, $w=2a$. By Lemma \ref{abelian-extension-properties} $a=\gamma^2$ for some $\gamma \in L$ such that $							[F(\gamma):F]=2$. Moreover, $\gamma \not\in F(i)$. Indeed, if it was then by Lemma \ref{quad} there will be $c\in F$ such that 
		$a = - c^2 $ but $-a = b$ is not a square in $F$ leading to a contradiction.
                We can rewrite $\beta^4+w^2=0$ as $(\beta^2-2ia)(\beta^2+2ia)=0$. Therefore, $\beta^2= \pm 2ia$ so that $\beta = \gamma ( i \pm 1)$ or $\beta= - \gamma ( i \pm 1)$. Proving the theorem since $\gamma/i$ is a generator of one of the quadratic subextensions of $L/F$ and $ i (i+1)= 1-i$. 

            \item We argue by contradiction. Suppose $i \in F$ and $L/F$ has a radical closure $K$.
                    By 1., $K(i)$ is quadratic extension of $L.K/K$ contradicting that
                    $i \in F \subset K$. Hence such $K$ does not exist, $L/F$ does not have a radical closure.

                 Conversely, we will prove that if $i \not\in F$ then $L/F$ has a radical closure.
                   For this, we suppose $i \not\in F$.  If $i \in L$ then $F$ is the trivial radical closure of $L/F$ and we are done.
                    We can then assume that $i \not\in L$, let $F(\gamma)$ be quadratic sub-extension of $L/F$ such that $\gamma^2 \in F$.
                    We claim that $K:=F(i\gamma)$ is a radical closure of $L/F$.
                    We first show that $L.K/K$ is a quartic extension, 
                    we prove by
                    contradiction that $K \cap L= F$.
                    Suppose $K \subset L$ then  $i\gamma, \gamma \in L$ which implies $i=\frac{i\gamma}{\gamma} \in L$,
                    contradicting that $L/F$ is not radical so indeed $K \cap L = F$ hence $L.K/K$ is a quartic extension.
                    It remains to show that $L.K/K$ is a radical extension, for this it suffices to show $K(i)$ is
                    quadratic sub-extension of $L.K/K$.
                    Note that $i \not\in K$, otherwise of $i \in K$ then $K=F(i)=F(i\gamma)$ and this implies
                    $\gamma \in F$ by Lemma \ref{quad} contradicting that $F(\gamma)$ is quadratic sub-extension of $L/F$.
                    We have $\gamma, i\gamma \in L.K$ hence $i \in L.K$ this implies $K(i)$ is indeed a quadratic
                    sub-extension of $L.K/K$, this proves $L.K/K$ is indeed a radical extension.

            \item  The "if direction" is a clear consequence of $1.$ and $2.$. We suppose that  $i \not\in L$. Since $ i \not\in F$, from the proof of $2.$, we know that
            $F(i\gamma)$ is a radical closure of $L/F$ whenever $F(\gamma)$ is a quadratic sub-extension of $L/F$.
            When $F(\gamma)$, $F(\delta)$  and $F(\gamma \delta)$ are distinct quadratic sub-extensions of $L/F$ then
            $F(i\gamma )$, $F(i\delta)$ and $F(i\gamma \delta)$ are also distincts. We can prove the latter as usual using Lemma \ref{quad}.  
Therefore, $L/F$ has at least $3$ radical closures. It remains to prove that given $M$ a radical closure of $L/F$, $M$ is one of the above.
   One can prove that $M(\gamma),M(\delta) \text{ and } M(\gamma\delta)$ are the quadratic sub-extensions of $L.M/M$.
       Also, by $1.$, since $L.M/M$ is radical, $M(i)$ is quadratic sub-extension of $L.M/M$. Thus, $M(i)=M( \theta )$ where $\theta \in \{ \gamma, \delta , \gamma \delta \}$ which implies $i\theta \in M$ by lemma  \ref{quad}. Hence, $M=F(i\theta )$, by definition of a radical closure. 
       This concludes the proof of $3.$ 
        \end{enumerate}
    \end{proof}

        \begin{remark}
    Note when $F$ is a function field and $L/F$ is a function fields' extension then the previous result state that if an
    extension is geometric and elementary abelian, then it cannot be radical.
    \end{remark}

    \begin{corollary}\label{cor-rad} 
              Let $L/F$ and $L'/F$ be two elementary abelian radical extensions with radical generators $\alpha \in L$
         and $\alpha' \in L'$ with minimal polynomial $P(X)=X^4+a$ and $Q(X)=X^4+a'$ respectively.
        Then, the following statements are equivalent:
        \begin{enumerate}
         \item $L/F$ is $F$-isomorphic to $L'/F$.
         \item $a'= d^4 a^j$ where $j$ is $1$ or $3$ and $d \in F$.
         \item $a'= c^4 a$ where $c \in F$.
         \item $ c\alpha' $ is a radical generator with minimal polynomial $P(X)$ where $c \in F$.
         \item $ d \alpha'^j$ is a radical generator with minimal polynomial $P(X)$ where $j$ is $1$ or $3$ and $d \in F$.
        \end{enumerate}
    \end{corollary}

    \begin{proof}

		$1. \implies 2.$ Let $\phi$ be a $F$-isomorphism from $L/F$ to $L'/F$.
		Let  $\alpha' = \phi(\alpha) \in L$. It is then a root of $Q(X)$. By Lemma \ref{elem_radical}  there exists radical generators $\gamma,\delta \in L$ of quadratic sub-extensions of $L/F$ different from $F(i)$ such that
		$\alpha=\gamma(i+1)$ and $\alpha'=\delta(i+1)$.
		$F(i),F(\gamma)$ and $F(i\gamma)$ are the quadratics sub-extensions of $L/F$ hence $F(\delta)$ is one of them i.e $F(\gamma)=F(\delta)$ or $F(i\gamma)=F(\delta)$, in the former we have
		 $\delta=d\gamma$ for some $d\in F$ by Lemma \ref{quad} hence $\alpha'=\delta(i+1)=d\gamma(i+1)=d\alpha$ so $a'=d^{4}a$.
		 When $F(\gamma)=F(i\delta)$  we have $\delta=b i\gamma$ for some $b \in F$  by Lemma \ref{quad} , let $e:=\gamma^2 \in F$.
		 Then $\alpha'=\delta(i+1)=bi\gamma(i+1)=-\frac{d}{2e}(\gamma(i+1))^3=d\alpha^3$ where $d=-\frac{c}{2e}$ hence $a'=d^{4}a^3$. 
		 
		$2. \implies 3.$
		When $j=1$ setting $c=d$ concludes the proof.
		From Lemma \ref{galois_closure} we have know that $a=\omega^2$ for some $\omega \in F$, so when $j=3$ we have $a'=d^{4}a^3=(d\omega)^{4}a$ then setting $c=d\omega$ concludes the proof.

		$3. \implies 4.$ $P(c\alpha' )=(c\alpha')^4+a'=c^4(-a')+a=0$, $c\alpha'$ being a primitive element of $L/F$ follows from the fact that $P(X)$ is irreducible over $F$.

		$4. \implies 5.$ Clear. 

		$5. \implies 1.$ The ring homomorphism sending $d\alpha'^j$ to $\alpha$ and extended by linearity is an $F$-isomorphism.

    \end{proof}

    \begin{lemma}\label{class_by_sub}
        Let $L/F$ be an elementary abelian extension, $F(\gamma)$ and $F(\delta)$ be distinct quadratic sub-extensions of $L/F$.
        If $L'/F$ is quartic extension then $L/F$ and $L'/F$ are $F$-isomorphic if and only $F(\gamma)$ and $F(\delta)$
        are isomorphic to quadratic sub-extensions of $L'/F$.
    \end{lemma}

    \begin{proof}
        $\Longrightarrow$ Let $\psi:L \rightarrow L'$ be a $F$-isomorphism, it is not hard to show that $F(\psi(\gamma))$ and $F(\psi(\delta))$
        are quadratic sub-extensions of $L'/F$ and that $F(\gamma) \cong F(\psi(\gamma))$ and $F(\delta) \cong F(\psi(\delta))$.

        $\Longleftarrow$ Let $L_1, L_2 \subset L$ be distinct quadratic extensions and $L_1', L_2' \subset L'$
        be distinct quadratic extensions such that $L_1:=F(\gamma) \cong L_1'$ and $L_2:=F(\delta) \cong L_2'$.        
        Let $\sigma_i : L_i \rightarrow L_i'$ be $F$-isomorphisms. 
       Since $\delta$ is a generator of $L_2$ over $F$, then $\sigma_2 (\delta)$ is a generator of $L_2'$ over $F$ both with the same minimal polynomial $P(X)$. 
        We have that $P(X)$ is irreducible over $F(\gamma)$. Indeed, $P(X)$ does not have a root in $F(\gamma)$ since $F(\gamma) \cap F(\delta )=F$.
        Hence, by \cite[Lemma 50]{rotman1}, there exists a $F$-isomorphism $\overline{\sigma_1}: L=L_1(\delta) \rightarrow L'=L_1'(\sigma_2(\delta))$ extending $\sigma_1$.
        This concludes the proof.
    \end{proof}

The following theorem, permits to classify elementary abelian extensions comparing their minimal biquadratic polynomials.

\begin{theorem}\label{elem_class}
    Let $L/F$ be an elementary abelian field extension, $y \in L$ be a biquadratic generator with minimal polynomial
    \[P(X)=X^4+uX^2+w^2 \in F[X].\]
    Let $L'/F$ be a biquadratic field extension, $y' \in L'$ biquadratic generator then with minimal polynomial
    \[Q(X)=X^4+vX^2+z^2 \in F[X]\]following
    then $L/F$ and $L'/F$ are $F$-isomorphic if and only if at least one the following statements is true.
  \begin{multicols}{3}
      \begin{enumerate}
          \item $\frac{-v-2z}{-u+2w}, \frac{-v+2z}{-u-2w} \in F^2$
          \item $\frac{-v-2z}{-u+2w}, \frac{-v+2z}{u^2-4w^2} \in F^2$

          \item $\frac{-v-2z}{-u-2w}, \frac{-v+2z}{-u+2w} \in F^2$
       	\item $\frac{-v-2z}{-u-2w}, \frac{-v+2z}{u^2-4w^2} \in F^2$
        \item $\frac{-v-2z}{u^2-4w^2}, \frac{-v+2z}{-u-2w} \in F^2$
        \item $\frac{-v-2z}{u^2-4w^2}, \frac{-v+2z}{-u+2w} \in F^2$
    \end{enumerate}\end{multicols}

\end{theorem}

\begin{proof}
    The quadratic sub-extensions of $L/F$ are precisely $F(\alpha)$, $F(\beta)$ and $F(\alpha\beta)$ with $\alpha^2=-u+2w$,
    and $\beta^2=-u-2w$ so $(\alpha\beta)^2=u^2-4w^2$, also the quadratic sub-extensions of $L'/F$ are $F(\delta)$, $F(\gamma)$
    and $F(\delta\gamma)$ where $\delta^2=-v+2z$, $\gamma^2=-v-2z$ and $(\delta\gamma)^2=v^2-4z^2$.
    By Lemma \ref{class_by_sub}, $L/F$ and $L'/F$ are isomorphic if and only if $F(\delta)$ and $F(\gamma)$
    are isomorphic to quadratic sub-extensions of $L/F$ since $L'=F(\delta).F(\gamma)$.
    There are exactly $6$ ways in which this can happen, these are the possibilities.
    \begin{enumerate}
          \item $F(\gamma) \cong F(\alpha)$ and $F(\delta) \cong F(\beta) \iff  \frac{\gamma^2}{\alpha^2}=\frac{-v-2z}{-u+2w},\frac{\delta^2}{\beta^2}= \frac{-v+2z}{-u-2w}, \in F^2$. 
               \item $F(\gamma) \cong F(\alpha)$ and $F(\delta) \cong F(\alpha\beta) \iff \frac{\gamma^2}{\alpha^2}=\frac{-v-2z}{-u+2w}, \frac{\delta^2}{(\alpha\beta)^2}=\frac{-v+2z}{u^2-4w^2} \in F^2$. 

        \item $F(\gamma) \cong F(\beta)$ and $F(\delta) \cong F(\alpha) \iff \frac{\gamma^2}{\beta^2}=\frac{-v-2z}{-u-2w},\frac{\delta^2}{\alpha^2}= \frac{-v+2z}{-u+2w} \in F^2$. 

        \item $F(\gamma) \cong F(\beta)$ and  $F(\delta) \cong F(\alpha\beta) \iff \frac{\gamma^2}{\beta^2}= \frac{-v-2z}{-u-2w}, \frac{\delta^2}{(\alpha\beta)^2}= \frac{-v+2z}{u^2-4w^2} \in F^2$. 

        \item $F(\gamma) \cong F(\alpha\beta)$ and $F(\delta) \cong F(\alpha) \iff \frac{\gamma^2}{(\alpha\beta)^2}=\frac{-v-2z}{u^2-4w^2}, \frac{\delta^2}{\alpha^2}= \frac{-v+2z}{-u+2w} \in F^2$. 
                \item $F(\gamma) \cong F(\alpha\beta)$ and $F(\delta) \cong F(\beta) \iff \frac{\gamma^2}{(\alpha\beta)^2}=\frac{-v-2z}{u^2-4w^2}, \frac{\delta^2}{\beta^2}=\frac{-v+2z}{-u-2w} \in F^2$. 
    \end{enumerate}
    The six previous equivalences are obtained using Lemma \ref{quad}
\end{proof}

The next remark explicit the relationship between generator of isomorphism elementary abelian extensions. 

\begin{remark} \label{elem_generators}
    \begin{enumerate}
        \item Note that for $P(X)=X^4+uX^2+w^2$ (resp. $Q(X)=X^4+vX^2+z^2$), we have  $P(X)=X^4-2(a+b)X^2+(a-b)^2$ where $a=-u-2w, b=-u+2w$
        (resp.  $Q(X)=X^4-2(a'+b')X^2+(a'-b')^2$ where $a'=-v-2z$ and $b'=-v+2z$).
        We also note that such $a, b$ (resp. $a', b'$) are unique (see also Theorem \ref{abelian-extension-properties}).
        \item  Let $L/F$ be an elementary abelian extension, $y$ a generator and $P(X)=X^4-2(a+b)X^2+(a-b)^2$ be its minimal polynomial (see Theorem \ref{abelian-extension-properties}). We also know from Theorem \ref{abelian-extension-properties}, 
        that $y :=\alpha +\beta$ whith $\alpha$ (resp. $\beta$) such that $\alpha^2=a$ (resp. $\beta^2 = b$).
        From the proof of Lemma \ref{abelian-extension-properties} and Theorem \ref{elem_class}, we have that the set of biquadratic generators of
        $L/F$ is $X:=X_{a} \cup X_{b} \cup X_{ab}$ where $X_{ab}=\{s \alpha +r \beta |s,r \in F^\times\},
        X_{a}=\{r \beta +s \alpha \beta |r,s \in F^\times\}, X_{b}=\{r\alpha +s \alpha \beta |r, s \in F^\times\}$.
        From $y^2=a+b+2\alpha\beta$, we get $\alpha\beta=\frac{1}{2}(y^2-a-b)$. So
        $\frac{1}{2}y(y^2-a-b)-ay= \alpha \beta y - a y =(a-b)\alpha$. Hence $\alpha=\frac{3a+b}{2(a-b)}y-\frac{1}{2(a-b)}y^3$. Similarly,
        $\beta=\frac{1}{2(a-b)}y^3-\frac{(a+3b)}{2(a-b)}y$.
        It follows that $s\alpha+r\beta=\frac{r-s}{2(a-b)}y^3+\frac{(3a+b)s-(a+3b)r}{2(a-b)}y$,
        $s\alpha+r\alpha\beta=s(\frac{3a+b}{2(a-b)}y-\frac{1}{2(a-b)}y^3)+r(\frac{1}{2}(y^2-a-b))=
        -\frac{(a+b)r}{2}+\frac{(3a+b)s}{2(a-b)}y+\frac{r}{2}y^2-\frac{s}{2(a-b)}y^3$, and
        $s\beta+r\alpha\beta=s(\frac{1}{2(a-b)}y^3-\frac{(a+3b)}{2(a-b)}y)+r(\frac{1}{2}(y^2-a-b))=
        -\frac{(a+b)r}{2}-\frac{(a+3b)s}{2(a-b)}y+\frac{r}{2}y^2+\frac{s}{2(a-b)}y^3$.
        Therefore, $X=\{ -\frac{(a+b)r}{2}-\frac{(a+3b)s}{2(a-b)}y+\frac{r}{2}y^2+\frac{s}{2(a-b)}y^3|r,s \in F^\times\} \cup
        \{-\frac{(a+b)r}{2}+\frac{(3a+b)s}{2(a-b)}y+\frac{r}{2}y^2-\frac{s}{2(a-b)}y^3|r,s \in F^\times\} \cup
        \{\frac{r-s}{2(a-b)}y^3+\frac{(3a+b)s-(a+3b)r}{2(a-b)}y| r,s \in F^\times\}$.
            Note that $X_{ab}=\{uy+vy^3| u,v \in F \text{ such that } u \neq -(a+3b)v  \text{ or } u  \neq -(3b +a) v \}$,
            this follows from the fact that $uy+v y^3=\frac{r-s}{2(a-b)}y^3+\frac{(3a+b)s-(a+3b)r}{2(a-b)}y$
            where $s=u+(a+3b)v$ and $r=u+(3a+b)v$.

    \end{enumerate}

\end{remark} 

The goal of the remaining part of this section is to reformulate the previous theorem into a group theoretic language. 

\begin{definition} \label{simrel}
    We define a relation $\sim$ on $F^\times \times F^\times$ as follows,
    if $(a,a'), (b,b') \in F^\times \times F^\times$ then $(a,a') \sim_1 (b,b')$ if only if at least one of the following is true
    \begin{multicols}{3}
        \begin{enumerate}
            \item $ab, a'b' \in F^2$
            \item $ab', a'b \in F^2$
            \item $ab, a'bb' \in F^2$
            \item $abb', a'b \in F^2$
            \item $ab', a'bb' \in F^2$
            \item $abb', a'b' \in F^2$
        \end{enumerate}
    \end{multicols}
\end{definition}

\begin{definition}
	Let $\sim_1$ be a relation on $F^\times \times F^\times$ defined by 
    $(a,a') \sim_1 (b, b')$ if $ab,a'b' \in F^2$.
	This relation is clearly an equivalence relation.
    Note that $(F^\times \times F^\times)/ \sim_1 \simeq F^\times / (F^\times)^2 \times F^\times / (F^\times)^2$
\end{definition}

\begin{remark} \label{eqre}
    \begin{enumerate}
        \item Note that $\sim$ is also the equivalence relation defined by $(a,a') \sim (b,b')$ if and only if at
            least one of the following is satisfied
    \begin{enumerate}
        \item[1.] $ (a,a') \sim_1 (b, b')$,
        \item[2.] $ (a,a') \sim_1 (b', b)\Leftrightarrow (a',a) \sim_1 (b, b')$
        \item[3.] $ (a,a') \sim_1 (b, bb')\Leftrightarrow (a,aa') \sim_1 (b, b')$
        \item[4.] $ (a,a') \sim_1 (bb', b)\Leftrightarrow (a',aa') \sim_1 (b, b')$
        \item[5.] $ (a,a') \sim_1 (b', b b') \Leftrightarrow (aa',a) \sim_1 (b, b')$
        \item[6.] $ (a,a') \sim_1 (bb', b')\Leftrightarrow (aa',a') \sim_1 (b, b')$ \end{enumerate}
    where each of those point are respectively equivalent to the point in the Definition \ref{simrel}.
    \item $(a, a') \sim (b, b') \Leftrightarrow (a,a' ) \sim (b' , b) \Leftrightarrow (a, a') \sim (b, bb') \Leftrightarrow (a, a') \sim (b', bb') \Leftrightarrow (a, a') \sim (bb', b) \Leftrightarrow (a, a') \sim (bb', b')$.
    \end{enumerate}
\end{remark}

\begin{lemma}
    The relation $\sim$ is an equivalence relation.
\end{lemma}

\begin{proof}
    Let $(a,a'), (b,b'), (c,c') \in  F^\times \times F^\times$
    \begin{itemize}

        \item Reflexitivity is clear
        \item Symmetry is not hard to prove from Remark \ref{eqre} since $\sim_1$ is symmetric.


        \item Let $(a,a') \sim (b,b')$ and $(b,b') \sim (c,c')$. From Remark \ref{eqre}, $(a,a') \sim_1 (b_1, b_1')$ and $(c,c')\sim_1 (b_2, b_2')$ where $b_i \in \{ b , bb', b'\}$ and $b_i' \in \{ b , bb', b'\}\backslash \{ b_i\}$ for $i=1,2$. Then, it is not hard to prove that $(b_1, b_1') \sim_1 ( c_1, c_1')$ where $c_1 \in \{ c , cc', c'\}$ and $c_1' \in \{ c , cc', c'\}\backslash \{ c_1\}$. 
        This will conclude the proof of transitivity.
    \end{itemize}
\end{proof}

\begin{lemma}[Definition]\label{s_3_action}
One can define an action of $S_3$ on $(F^\times \times F^\times)/ \sim_1\simeq F^\times/ (F^\times)^2  \times F^\times / (F^\times)^2$. Given $\sigma \in S_3$ and  $[(b, b')]_{\sim_1}\in (F^\times \times F^\times)/ \sim_1 $, we define $\sigma \cdot [(b, b')]_{\sim_1} := [(\sigma_{(b,b')} (b) , \sigma_{(b,b')} (b'))]_{\sim_1}$ where $\sigma_{(b,b')}(b) = c $ and $\sigma_{(b,b')}(b')=d$ where $c, d \in \{ b, b' , bb'\}$ such that when $b$ is identified with $1$, $b'$ with $2$ and $bb'$ with $3$, $c$ (resp. d) is the element of $\{ b, b' , bb'\}$ identified with $\sigma(1)$ (resp. $\sigma(2)$).     We will denote the set of equivalence classes by $\frac{ F^\times/ (F^\times)^2  \times F^\times / (F^\times)^2}{S_3}$ and $O_{S_3}([(a,b)]_{\sim_1})$ the orbit of $[(a, b)]_{\sim_1}$ via this action.

\end{lemma}
\begin{proof}
    We start by showing that given $\sigma \in S_3$ and $(a,b)\in F^\times \times F^\times$, we to show that $\sigma([(a,b)]_{\sim_1})$ is well defined i.e
    $\sigma([(a,b)]_{\sim_1})$ does not depend on the choice of $(a,b)$. Let $[(a,b)]_{\sim_1}=[(c,d)]_{\sim_1}$ then $ac,bd \in F^2$ hence
    $(ab)(cd) \in F^2$, now we want to prove that $[(a',b')]_{\sim_1} = [(c',d')]_{\sim_1}$
    where $a'= \sigma_{[(a,b)]} ( a)$, $b' =  \sigma_{[(a,b)]} ( b)$, $c'= \sigma_{[(c,d)]} ( c)$ and $d'= \sigma_{[(c,d)]} ( d)$, then \begin{itemize}
             \item if $\sigma(1)=1$ then $a'=a$ and $c'=c$ hence $a'c'=ac \in F^2$
             \item if $\sigma(1)=2$ then $a'=b$ and $c'=d$ hence $a'c'=bd \in F^2$
             \item if $\sigma(1)=3$ then $a'=ab$ and $c'=cd$ hence $a'c'=(ab)(cd) \in F^2$
    \end{itemize}
    Similar arguments can be used to show $b'd' \in F^2$ by looking at possibilities of $\sigma(2)$, this proves
    $[(a',b')]_{\sim_1}=[(c',d')]_{\sim_1}$ so $\sigma([(a,b)]_{\sim_1})$ is indeed well defined.
	Let $\psi : S_3 \times (F^\times \times F^\times)/ \sim_1 \rightarrow (F^\times \times F^\times)/ \sim_1 $ be
    defined by $\psi(\sigma,[(a,b)]_{\sim_1})=\sigma([(a,b)]_{\sim_1})$. The fact that $\psi$ is a group action follows from the fact that $S_3$ acts on $F^\times \times F^\times$. 

\end{proof}

\begin{lemma}\label{corr}
 We have the following natural bijection:
 $$(F^\times \times F^\times)/ \sim \simeq \frac{ F^\times/ (F^\times)^2  \times F^\times / (F^\times)^2}{S_3} $$
\end{lemma}

\begin{proof}
        Let $\phi: (F^\times \times F^\times)/ \sim  \rightarrow \frac{ F^\times/ (F^\times)^2  \times F^\times / (F^\times)^2}{S_3}$
        be defined by $\phi([(a,b)]_{\sim})=O_{S_3}([(a,b)]_{\sim_1})$. We start by showing that $\phi$ is well defined.
        Let $[(a,b)]_\sim=[(c,d)]_\sim$. Then, by Remark \ref{eqre}, there exists $c',d' \in \{c,d,cd\}$ such that $[(a,b)]_{\sim_1}= [(c',d')]_{\sim_1}$ where $c'=\sigma_{(c,d)}(c)$, $d'=\sigma_{(c,d)}(d)$ for some $\sigma \in S_3$. Hence $\phi([(a,b)]_\sim)=
        O_{S_3}([(a,b)]_{\sim_1})=O_{S_3}([(c,d)]_{\sim_1})=\phi([(c,d)]_\sim)$ proving that $\phi$ is well defined.
        It's clear that $\phi$ is surjective, so it remains to show that $\phi$ is injective.
        Let $\phi([(a,b)]_\sim)=\phi([(a,b)]_\sim)$ that is $O_{S_3}([(a,b)]_{\sim_1})=O_{S_3}([(c,d)]_{\sim_1})$, this implies there exists $\sigma \in S_3$ such that
        $[(a,b)]_{\sim_1}= [(\sigma_{(c,d)}(c), \sigma_{(c,d)}(d))]_{\sim_1}=[(c',d')]_{\sim_1}$ where $c',d' \in \{c,d,cd\}$ and by Remark
        \ref{eqre} we have that $[(a,b)]_\sim=[(c,d)]_\sim$ hence $\phi$
        injective.
\end{proof}

We are aiming to use the above quotient to describe a bijective correspondance between the elementary abelian extension up isomorphism and a subset of this quotient. To do this, we can consider quartic elementary abelian extensions as compositum of two distinct quadratic extensions and identify them with the pair of corresponding parameters of those quadratics extension for some choosen radical generators. Following this identification, we need to exclude the following subset
$$S=\{ [ (a, a') ]_\sim | a \in (F^\times)^2 \ or \ a' \in (F^\times)^2\ or \ aa' \in (F^\times)^2\}$$
as it leads to a compositum of degree smaller than $4$. 
\begin{lemma}\label{Nelab} 
We have that
$$ S=  \{[(1,a)]_\sim | a \in F^\times\} \simeq \frac{ N}{ S_3}\simeq F^\times/ (F^\times)^2$$
where $N$ is a subgroup $F^\times/ (F^\times)^2  \times F^\times / (F^\times)^2$ isomorphic to $ F^\times/ (F^\times)^2$.
\end{lemma}
\begin{proof}
Note that $1$ is a representative for a square in $F/(F^\times)^2$ and therefore 
$$\begin{array}{cll} S &=& \{[1,a]_{\sim}| a \in F^\times \}\cup \{[a,1]_{\sim}| a \in F^\times \}\ \cup \{[a,1/a]_{\sim}| a \in F^\times \}\\ &=&  \{[1,a]_{\sim}| a \in F^\times \}
= \{[a,1]_{\sim}| a \in F^\times \} = \{[a,1/a]_{\sim}| a \in F^\times  \} \end{array}$$
We can take $N=\{[1,a]_{\sim_1}| a \in F^\times \}$ (Note that we could have taken also to be either $\{[a,1]_{\sim_1}| a \in F^\times \}$ or$\{[a,1/a]_{\sim_1}| a \in F^\times \}$) and $\psi: \frac{N}{S_3} \rightarrow F/(F^\times)^2$ be defined as 
$\psi(O_{S_3}([(1,a)]))=a(F^\times)^2$. We will show that $\psi$ is a group isomomorphsim. It is clear that 
 $\psi$ is well defined. Indeed, by definition, $O_{S_3}([(1,a)]_{\sim_1})=O_{S_3}([(1,b)]_{\sim_1})$ implies 
 $a(F^\times)^2=b(F^\times)^2$ so $\psi$ is indeed well defined. The fact that $\psi$ is surjectivite and a group homormorphism is clear. The injectivity is the result of the definition of $\sim$ and Lemma \ref{corr}.

\end{proof}
\begin{theorem}\label{final_elem_class}
$F$ is fixed and using the notation above. We have the following bijective correspondance:
$$\{ L/F \text{ elementary abelian extension}\} / \sim_{iso} \simeq  \frac{ [(F^\times/ (F^\times)^2  \times F^\times / (F^\times)^2]- N }{S_3}  $$
where $\sim_{iso}$ is the equivalence relation on the quartic extensions $L/F\sim_{iso} L'/F$ if $L/F$ is $F$-isomorphic to $L'/F$.
\end{theorem}

\begin{proof}
	Let us define the map $\psi: \{ L/F \text{ elementary abelian extension}\} / \sim_{iso} \rightarrow \frac{ (F^\times/ (F^\times)^2  \times F^			\times / (F^\times)^2)\backslash N }{S_3}$.
For $L/F$ be an elementary abelian extensions and $\alpha \in L$ be a biquadratic generator, then by Lemma 
	\ref{abelian-extension-properties} the minimal polynomial of $\alpha$ over $F$ is $P(X)=X^4-2(a+b)X^2+(a-b)^2$ for some 
	$a,b \in F$ moreover we have $\gamma,\delta \in L$ such that $a=\gamma^2$, $b=\delta^2$ and $F(\gamma)$, $F(\delta)$,
	$F(\gamma\delta)$ are the quadratic sub-extensions of $L/F$.
	We define $\psi$ as  $\psi([L/F]_{iso})=O_{S_3}([(a,b)])$. 
	One can prove without much difficulty that $\psi$ is well defined injection using Theorem \ref{elem_class} and Lemma \ref{corr}. 
	 It remains to show that $\psi$ is surjective. For any $O_{S_3}([a,b]) \in \frac{ (F^\times/ (F^\times)^2  \times F^\times / (F^\times)^2)\backslash N }{S_3}$, we have $\phi(F(\alpha)) = O_{S_3}([a,b])$ where $\alpha$ is a root of $P(X)=X^4-2(a+b)X^2+(a-b)^2$ irreducible over $F$ by Lemma \ref{irred} since $a,b, \text{ and }ab$ are not squares in $F$, by Lemma \ref{Nelab}. 
	
\end{proof}

\begin{remark}
\begin{enumerate}
\item Let $F$ be a field such that $-1 \not\in F^2$, then there exist a one to one correspondence between the isomorphism class quartic radical elementary abelian extensions
	and the set 
	$$\begin{array}{ccl}&& \{[(-1,a)]_{\sim}| a \in F^\times\} -(\{-1\} \times -( F^\times)^2 \cup \{-1\} \times ( F^\times)^2 ) \\&=&\{[(-1,a)]_{\sim}| a \in F^\times \} -\{[(-1,1)]_{\sim}, (-1,-1)]_{\sim} \}\\
	& \simeq & F^\times/ {F^\times}^2 - \{ -( F^\times)^2 , ( F^\times)^2 \}.\end{array} $$
	We note that the pair $(-1, 1)$  correspond to an extension defined by the minimal polynomial $X^4 +4$. 
\item When $F$ is a field and $-1 \in F^2$, we observe again that in this situation, there is no radical elementary abelian extension (see also Theorem \ref{elem_radical}).
\end{enumerate} 
\end{remark}

\subsection{Elementary abelian closure}\label{subsec:elementary-abelain-closure}

In this section, we study the existence and the uniqueness of the elementary abelian closure for biquadratic extension. We apply this to the classification of non-cyclic extensions. 

\begin{remark}\label{elem_min_poly_remains_rmk}
Let $L/F$ be a biquadratic extension admitting an elementary abelian closure $E$,
Then any biquadratic generator $\alpha$ of $L/F$ is a biquadratic generator of $LE/E$ i.e $LE=E(\alpha)$ and
$min(\alpha,F)=min(\alpha,E)$ where the second equality follows because $min(\alpha,F)$ and $min(\alpha,E)$ are both
degree $4$ monic elements of $E[X]$ having $\alpha$ as a root.
\end{remark}

\begin{lemma}\label{unique_elem_closure}
Let $L/F$ be a biquadratic field extension and $\alpha \in L$ be a biquadratic generator
with minimal polynomial $P(X)=X^4+uX^2+w$.
\begin{enumerate}
    \item $L/F$ admits a trivial elementary abelian closure if and only if $w \in F^2$.
    \item If $L/F$ admits an elementary abelian closure $E$, then $E=F(\eta)$ for some $\eta \in \overline{F}$ such
    that $\eta^2=w$.
    \item $L/F$ admits an elementary abelian closure if and only if $L/F$ is not cyclic.
\end{enumerate}
\end{lemma}

\begin{proof}
    It follows from Lemma \ref{inter} and Theorem \ref{abelian-extension-properties} that $L/F$ has a unique quadratic
    sub-extension $M=F(\gamma)$ for some $\gamma \in L$ such that $\gamma^2=u^2-4w$.
    \begin{enumerate}
        \item Follows from Theorem \ref{abelian-extension-properties}.
        \item
        Suppose $L/F$ admits an elementary abelian closure $E$.
        If $w \in F^2$ then $E=F$, also by Lemma \ref{abelian-extension-properties} we have that $L/F$ is elementary
        abelian, implying $L/F$ has trivial elementary abelian closure hence the statements is true.
        Let $L/F$ be an non-elementary abelian and $\eta$ be an element in $\overline{F}$ such that $\eta^2=w$.
        Note $\alpha$ is a biquadratic generator of $LE/E$ and $P(X)$ is the minimal polynomial by Remark
        \ref{elem_min_poly_remains_rmk}, hence $w \in E^2$ by Lemma \ref{galois_closure} and $F(\eta) \subseteq E$.
        It follows from the minimality of degree of $E$ over $F$ such that $E=F(\eta)$.

        \item
        $\Longrightarrow$
        Suppose $L/F$ does not admit an elementary abelian closure, then $L/F$ is not elementary abelian and $w \not\in F^2$
        by Lemma \ref{galois_closure}.
        Let $\eta \in \overline{F}$ such that $\eta^2=w$ then $E:=F(\eta)$ is a quadratic extension, $E \subset L$. Indeed,
        $P(X)$ is reducible over $E$ otherwise $LE=E(\alpha)$ is quartic extension over $E$ and elementary abelian
        by Lemma \ref{galois_closure}, contradicting that $L/F$ does not admit an elementary abelian closure.
        From Lemma \ref{irred} we know that at least one the following statements is true
        \begin{enumerate}
            \item $u^2-4w=-4(\beta^2-\sigma(\beta)^2)^2 \in E^2$
            \item $-u+2\sqrt{w}=2(\beta+\sigma(\beta))^2 \in E^2$
            \item $-u-2\sqrt{w}=-2(\beta-\sigma(\beta))^2 \in E^2$
        \end{enumerate}

        If $\delta \in E$ such that $\delta^2=-u-2\sqrt{w}$ (resp. $\delta^2=-u+2\sqrt{w}$) then $\delta$ is a root of
        $T(X)=X^4+2uX^2+u^2-4w$.
        $T(X)$ is irreducible over $F$ by Lemma \ref{irred} as $u^2-4w \not\in F^2$ and $(2u)^2-4(u^2-4w)=16w \not\in F^2$
        this implies $F(\delta)$ is a quartic extension of $F$ contradicting that $E$ is a quadratic extension of $F$.
        It follows that there exists $\theta \in E$ such that $\theta^2 = u^2-4w$, $F(\theta)$ is a quadratic extension
        as $u^2-4w \not\in F^2$ by Lemma \ref{irred}.
        Hence $F(\theta)=F(\eta)$ it follows from Lemma \ref{quad} that $\frac{u^2-4w}{w} \in F^2$ hence
        $w(u^2-4w) \in F^2$ proving that $L/F$ is cyclic by Lemma \ref{galois_closure}.

        $\Longleftarrow$
        Let $L/F$ be cyclic.
        Suppose $L/F$ admits an elementary abelian $E$, then $E=F(\eta)$ for some $\eta \in \overline{F}$ such that
        $\eta^2=w$ by $2.$
        Note that $L/F$ cyclic implies $w \in L^2$ by Lemma \ref{galois_w}, it follows that $E$ is $F$-isomorphic to the
        quadratic sub-extension of $L/F$ hence $P(X)$ is reducible over $E$ and $LE/E$ is not a quartic extension.
        This contradicts our assumption, hence $L/F$ does not admit an elementary abelian.

    \end{enumerate}
\end{proof}

\begin{theorem}\label{elem_closure_descent}
Let $L/F$ and $L'/F$ be biquadratic fields extensions admitting elementary abelian closure $E$ and $E'$ respectively.
If $\alpha$ is a biquadratic generator of $L/F$ and $\beta \in L'$ is a biquadratic generator of $L'/F$ such that
$min(\alpha,F)=P(X)=X^4+uX^2+w$ and $min(\beta, F)=Q(X)=X^4+vX^2+z$ then the following are equivalent
\begin{enumerate}
    \item $L/F \cong L'/F$
    \item $LE/E/F \cong L'E'/E'/F$
    \item There exist $a,c \in F$ and $s,r \in \{1,-1\}$ such that
    $\frac{z}{w}=c^2$, $ \frac{v^2-4z}{u^2-4w} =a^2$ and $\Omega= \frac{1}{2}\bigg(\frac{ra}{w} + \frac{uv-4wcs}{w(u^2-4w)}\bigg) \in F^2 $,  
    when $\Omega =0$, in addition, either $\frac{ra}{w} \in F^2$ or  $ra \in F^2$.

\end{enumerate}
\end{theorem}

\begin{proof}

    By Lemma \ref{unique_elem_closure}, $E=F(\eta)$ for some $\eta \in E$ such that $\eta^2=w$,
    similarly $E'=F(\epsilon)$ for some $\epsilon \in E'$ such that $\epsilon^2=z$.

    $1. \implies 2.${change}
    The isomorphism $L/F \cong L'/F$ extends to a isomorphism of $LE/E/F \cong L'E'/E'/F$, this follows from the uniqueness of the elementary abelian closure  

    $2. \implies 1.$ 
    Let $LE/E/F \cong LE'/E'/F$. 
    This implies that $Q(X)$ is the minimal polynomial of some primitive element in $LE$.
    To show that $L/F \cong L'/F$ it suffice to prove that $Q(X)$ has a root in $L$.
    Let $\pm \alpha, \pm \delta$ be the roots of $P(X)$ in $LE$ and $\pm \mu, \pm \kappa$ be the roots of $Q(X)$ in $LE$,
    such roots exist because $LE/E$ is Galois.
    Let $\sigma$ be the generator of $Gal(LE/L)$. We show that either $\sigma(\mu)=\mu$ or $\sigma(\kappa)=\kappa$.
    From $\epsilon^2=z=(\mu\kappa)^2$, $\eta^2=w=(\alpha\delta)^2$, $-v=\mu^2+\kappa^2$ and $-u=\alpha^2+\beta^2$ we
    get that there exist $s_1, s_2 \in \{1,-1\}$ such that $\mu\kappa=s_1 \epsilon$ and $\alpha\delta=s_2 \eta$, now by Lemma
    \ref{quad} there exist $c \in F$ such that $\epsilon=c\eta$ 
    so $\mu\kappa=c'\alpha\delta$ where $c'=c s_1 s_2$.
    Up to changing the choice $\epsilon$ (resp $\eta$) of the root of $X^2 -z$ (resp. $X^2-w$), we can assume $s_1=s_2=1$.
    So that $\sigma((\mu-\kappa)^2)=\sigma(-v-2\epsilon)=-v+2\epsilon=(\mu+\kappa)^2$
    hence $\sigma(\mu)-\sigma(\kappa)=\sigma(\mu-\kappa)= \pm (\mu+\kappa)$.
    Also $\sigma(\delta)=\sigma(\frac{ \eta}{\alpha})=-\frac{\eta}{\alpha}=-\delta$, hence
    $\sigma(\mu)\sigma(\kappa)=\sigma(\mu\kappa)=\sigma(c\alpha\delta)=-c\alpha\delta=-\mu\kappa$.
    If $\sigma(\mu)-\sigma(\kappa)=-(\mu +\kappa)$ then $\sigma(\mu)^2+(\mu+\kappa)\sigma(\mu)+\mu\kappa=0$
    so $\sigma(\mu)=-\mu$ or $\sigma(\mu)=-\kappa$.
    Note that $\sigma(\mu)=-\kappa$ is not possible otherwise
    $\sigma(\kappa)=-\mu$, then $\mu\kappa=c \eta $ is fixed by $\sigma$ contradicting that $w \not\in L^2$.
    Now $\sigma(\mu)=-\mu $ implies $\sigma(\kappa)=\kappa$ from $\sigma(\mu\kappa)=-\mu\kappa$.
    Similarly $\sigma(\mu)-\sigma(\kappa)=\mu+\kappa$ implies $\sigma(\mu)=\mu$, so indeed either $\kappa \in L$
    or $\mu \in L$ concluding the proof.

    $2. \Leftrightarrow 3.$
    Let $LE/E/F \cong L'E'/E'/F$. 
    Since $E/F \cong E'/F$ 
    is equivalent by Lemma \ref{quad} to the existence of $c \in F$ such that $\epsilon=c\eta$.
    By Theorem \ref{elem_class}, $LE/E \cong L' E'/E'$  is equivalent to one of the following statements being true
    \begin{enumerate}
        \item[(a)]  $\frac{-v-2c\eta }{-u-2s\eta} \in E^2$ for some $s \in \{1,-1\}$
        \item[(b)] $\frac{-v-2c\eta}{u^2-4w} \in E^2$
    \end{enumerate}
    We prove that $(b)$ cannot be true since $LE/E$ and $L' E'/E'$ are elementary abelian.
    We argue by contradiction, assuming that there exist $g,h \in F$ such that $\frac{-v-2c \eta}{u^2-4w}=(g+h\eta)^2$.
    Let $\sigma$ be generator of $Gal(E/F)$, then $\sigma(\frac{-v-2c \eta }{u^2-4w})=\frac{-v+2c \eta }{u^2-4w}=
    (g-h\eta )^2 \in E^2$
    hence $E(\sqrt{-v-2c \eta }) \cong E(\sqrt{-v+2c \eta}) \cong E(\sqrt{u^2-4w})$, by Lemma \ref{quad},
    contradicting that $E(\sqrt{-v-2 c \eta} ), E(\sqrt{-v+2 c \eta})$ are distinct quadratic sub-extensions of $LE/E$ by Theorem
    \ref{abelian-extension-properties}.
    It follows that $(a)$ is true, that is $m :=\frac{-v-2c\eta}{-u-2s\eta}=(x+y\eta)^2$ for some $x,y \in E$.
    That is
    $$\frac{(-v-2c\eta) (-u+2s\eta) }{u^2-4w} = x^2 +4w y^2 +2xy \eta$$
    Equivalently,
    $$ uv -4csw + 2  \eta (cu -sv) =  (x^2 +4w y^2)( u^2-4w)+2xy(u^2-4w) \eta$$
    This is true if and only if $(x^2 +w y^2)( u^2-4w)= uv -4csw  $ and $cu-sv = xy(u^2-4w)$  $( * )$.
    \begin{itemize}
        \item when $cu=sv$, we have $\frac{v^2-4z}{u^2-4w}= c^2$, one could have
        \begin{itemize}
            \item $x=0$, then $w y^2( u^2-4w) = uv-4csw$, that is $y^2 = \frac{uv-4csw}{w ( u^2-4w)}= \frac{c}{sw}$ (In this case, $w^3 y^4 = z$). 
            \item $y=0$, then $x^2= \frac{uv-4csw}{u^2-4w}= \frac{c}{s}$ ($w x^4= z$). 
         Note that taking $a:=c$, $r:=-1$ implies  $\frac{1}{2}\bigg(\frac{ra}{w} + \frac{uv-4wcs}{w(u^2-4w)}\bigg)=0$.
        \end{itemize}
        (Note that both cannot be true simultaneously otherwise $w$ would be a square, and $L/F$ is not elementary abelian.)
        \item when $cu \neq sv$, We have $x = \frac{ cu-sv }{ y(u^2-4w)}$ and
        $$y^4 + \frac{4wcs-vu}{w(u^2-4w)} y^2 + \frac{1}{w}\bigg(\frac{vs-uc}{4w-u^2} \bigg)^2 =0 ,$$
        so $y$ is a root of $S(X)=X^4+(\frac{4wcs-vu}{w(u^2-4w)})X^2 +\frac{1}{w}\bigg(\frac{vs-uc}{4w-u^2} \bigg)^2$.\\ \\
        Let $\Delta=\bigg(\frac{uv-4wcs}{w(u^2-4w)}\bigg)^2-4\bigg(\frac{1}{w}\big(\frac{vs-uc}{4w-u^2}\big)^2\bigg)
        =\frac{a^2}{w^2}$
        then $y$ is such that $$y^2=\frac{1}{2}\bigg(\pm \sqrt{\Delta} + \frac{uv-4wcs}{w(u^2-4w)}\bigg)=
        \frac{1}{2}\bigg(\pm \frac{a}{w} + \frac{uv-4wcs}{w(u^2-4w)}\bigg).$$
        Hence we have that either $\frac{1}{2}\bigg( \frac{a}{w} + \frac{uv-4wcs}{w(u^2-4w)}\bigg) \in F^2$ or
        $\frac{1}{2}\bigg(- \frac{a}{w} + \frac{uv-4wcs}{w(u^2-4w)}\bigg) \in F^2$.

        Proving 3.
        (Note that given $L/F$ as in the statement we cannot have \\
        $\frac{1}{2}\bigg( \frac{a}{w} + \frac{uv-4wcs}{w(u^2-4w)}\bigg) \in F^2$  and
        $\frac{1}{2}\bigg(- \frac{a}{w} + \frac{uv-4wcs}{w(u^2-4w)}\bigg) \in F^2$.
        Indeed, if that happens since
        $$ \frac{1}{2}\bigg( \frac{a}{w} + \frac{uv-4wcs}{w(u^2-4w)}\bigg)\frac{1}{2}\bigg(- \frac{a}{w} + \frac{uv-4wcs}{w(u^2-4w)}\bigg) = \frac{1}{w}\bigg(\frac{vs-uc}{4w-u^2}\bigg)^2$$ then $w$ would be a square in $F$ which is excluded in the assumption.)

        $3. \Rightarrow 2.$ Given $a,c,d \in F$ and $s,r \in \{1,-1\}$ such that
        $\frac{z}{w}=c^2$, $\frac{v^2-4z}{u^2-4w} =a^2$
        and $\frac{1}{2}\bigg(\frac{ra}{w} + \frac{uv-4wcs}{w(u^2-4w)}\bigg) \in F^2$, we are looking for $x$ and $y$ in
        $F$ satifying $(*)$. When $\Omega \neq 0$, we can take $y$ such that
        $\frac{1}{2}\bigg(\frac{ra}{w} + \frac{uv-4wcs}{w(u^2-4w)}\bigg)  = y^2$ and
        $x= \frac{ cu-sv }{ y(u^2-4w)}$ and when $\Omega=0$, $\frac{ra}{w} + \frac{uv-4wcs}{w(u^2-4w)}=0$ then
        $a = - \frac{uv-4wcs}{r(u^2-4w)}$  so that $$ \begin{array}{lll} a^2 = \frac{v^2 -4z}{u^2 -4w} =\bigg( \frac{uv-4wcs}{r(u^2-4w)}\bigg)^2\\ \\
        \Leftrightarrow  (v^2 -4z)(u^2 -4w) = u^2 v^2 -8uvwcs + 16w^2 c^2 \\
        \Leftrightarrow u^2 v^2 -4zu^2 -4w v^2+ 16 wz  =  u^2 v^2 -8uvwcs + 16wz \\
        \Leftrightarrow zu^2 +w v^2= 2uvwcs\\
        \Leftrightarrow  \frac{z}{w} \big( \frac{u}{v}\big)^2-2 \frac{u}{v} c s  +1 \\
        \Leftrightarrow  \big(c \frac{u}{v}\big)^2-2 \frac{u}{v} c s  +1  \\
        \Leftrightarrow  ( c \frac{u}{v} -s)^2 =0
        \Leftrightarrow cu = sv
        \end{array}$$
        If $ c r \in F^2$, then $x$ such that $x^2 =cr$ and $y=0$ would be pairs of solution for $(*)$.
        Finally, if $\frac{cr}{w}\in F^2$, then $y$ such that $y^2 = \frac{cr}{w}$ and $x=0$ would be pairs of solution for $(*)$ concluding the proof of the theorem.

    \end{itemize}
\end{proof}

\begin{remark}
    Let $L/F$ be a non-elementary abelian field extension, $y \in L$ be a biquadratic generator over $F$ with
    minimal polynomial $P(X)=X^4+uX^2+w$.
    Let $E:=F(\eta)$ where $\eta \in \overline{F}$ such that $\eta^2=w$, then by Lemma \ref{unique_elem_closure} we have
    that $E$ the elementary abelian closure of $L/F$, so $LE/E$ is an elementary abelian extension and $y$ is a
    biquadratic generator of $LE/E$.
    Note that $P(X)=X^{4}-2(a+b)X^2+(a-b)$ where $a=\frac{1}{4}(-u-2\eta)$ and $b=\frac{1}{4}(-u+2\eta)$. Moreover by
    Theorem \ref{abelian-extension-properties} there exists $\alpha, \beta \in LE$ such that $\alpha^2=a$, $\beta^2=b$
    and $y=\alpha+\beta$.
    By Remark \ref{elem_generators} we have that any biquadratic generator of $LE/E$ is in $X:=X_{a} \cup X_{b} \cap
    X_{ab}$ where
    $X_{a}=\{ -\frac{(a+b)r}{2}-\frac{a+3b}{2(a-b)}y+\frac{r}{2}y^2+\frac{s}{2(a-b)}y^3|r,s \in E^\times\}$,
    $X_{b}=\{-\frac{(a+b)r}{2}+\frac{(3a+b)s}{2(a-b)}y+\frac{r}{2}y^2-\frac{s}{2(a-b)}y^3|r,s \in E^\times\}$ and
    $X_{ab}=\{\frac{r-s}{2(a-b)}y^3+\frac{(3a+b)s-(a+3b)r}{2(a-b)}y | r,s \in E^\times\}$.
    We want to show that all biquadratic generators of $L/F$ are in $ X':=\{s\alpha+r\beta| r,s \in E^\times \text{ and }
    \sigma(s)=r\} \subset X_{ab}$ where $\sigma$ is a generator of $LE/L$.
    Note that $\sigma(a)=b$ i.e $\sigma(\alpha)^2=\sigma(\beta)^2$ hence $\sigma(\alpha)=\beta$ or $\sigma(\alpha)=-\beta$.
    Suppose $\sigma(\alpha)=-\beta$ then $\sigma(y)=\sigma(\alpha)+\sigma(\beta)=-(\alpha+\beta)=-y$ contradicting that
    $y \in L$ hence fixed by $\sigma$, so $\sigma(\alpha)=\beta$.
    Let $y' \in LE$ be a biquadratic generator of $LE/E$. 
    If $y \in X_{a}$ then $y'=s\beta+r\alpha\beta$ for some $s,r \in E^\times$ so $\sigma(y')=\sigma(s)\alpha+\sigma(r)\alpha\beta \in X_{b}$
    hence $\sigma(y') \neq y'$ so $y' \not\in L$. Similarly if $y' \in X_{b}$ then $y' \not\in L$.
    Now let $y' \in X_{ab}$ then $y'=s\alpha+r\beta$ for some $r,s \in E^\times$, so $\sigma(y')=\sigma(r)\alpha+
    \sigma(s)\beta$, it follows that $y' \in L$ if and only if $\sigma(s)=r$, that is if and only if $y' \in X'$
    concluding the proof.

\end{remark}

\begin{definition}
   Let $E/F$ be a quadratic field extension, then $Gal(E/F) \cong S_2$.
    Let $\phi: S_2 \rightarrow Gal(E/F)$ be a group isomorphism, then $S_2$ act on $\frac{E^\times}{(E^\times)^2}$ by the group action
    $\tau: S_2 \times \frac{E^\times}{(E^\times)^2} \rightarrow \frac{E^\times}{(E^\times)^2}$ as $\tau(\sigma, x)=
    \phi(\sigma)(x)(E^\times)^2$.
\end{definition}

\begin{lemma}\label{final_non_Galois_class}
     Let $F$ be field and $S \subset F^\times$ be a complete set of representatives of $\frac{F^\times}{(F^\times)^2}-{\{(F^\times)^2 \}}$.
    Let $s \in S$
        and $N_{F(\sqrt{s})}=\{\alpha \in ({F(\sqrt{s})})^\times| \alpha\sigma(\alpha) \in (F(\sqrt{s})^\times)^2\}$
    where $\sigma$ is the generator of  $Gal(F(\sqrt{s})/F)$.
    Then $$\{ L/F \text{ non-Galois quartic extension}\} / \sim_{iso} \simeq
\coprod_{s \in S} \frac{\frac{F(\sqrt{s})^\times - N_{F(\sqrt{s})}}{(F(\sqrt{s})^\times)^2}}{S_2}$$
\end{lemma}

\begin{proof}
    Let $\psi:    \coprod_{s \in S} \frac{\frac{F(\sqrt{s})^\times - N_{F(\sqrt{s})}}{(F(\sqrt{s})^\times)^2}}{S_2} \rightarrow
    \{ L/F \text{ non-Galois quartic extension}\} / \sim_{iso}$ be defined as
    $$\psi([(a+b\sqrt{s})(F(\sqrt{s})^\times)^2]_{S_2})=[F(\beta)]_{iso}$$ where $\beta$ is a root of
    $P(X)=X^4-aX^2+\frac{1}{4}b^{2}s$.
   Indeed, by Lemma \ref{irred}, we have $F(\beta)$ is a biquadratic extension since $\alpha:=a+b\sqrt{s} \notin N_{F(\sqrt{s})}$.
    Moreover, $s \not\in F^2$ hence $F(\beta)$ is not elementary abelian by Lemma \ref{galois_closure}.
  In addition,  $\alpha \not\in N_{F(\sqrt{s})}$ implies $\frac{1}{4}b^{2}s(a^2-b^{2}s) \notin F^2$. That is $F(\beta)$ is not a cyclic extension. As a consequence, $F(\beta)$ is non-Galois.

    Now we show that $\psi$ is well-defined.     Let $[(a+b\sqrt{s})(F(\sqrt{s})^\times)^2]_{S_2},[(a'+b'\sqrt{s'})(F(\sqrt{s'})^\times)^2]_{S_2} \in \amalg_{s \in S} \frac{\frac{F(\sqrt{s})^\times - N_{F(\sqrt{s})}}{(F(\sqrt{s})^\times)^2}}{S_2}$
    such that  $[(a+b\sqrt{s})(F(\sqrt{s}))^2]_{S_2}=[(a'+b'\sqrt{s'})(F(\sqrt{s'}))^2]_{S_2}$.
    We want to prove that $ [F(\beta)]_{iso}= [F(\beta')]_{iso}$ where $\beta,\beta' \in \overline{F}$ are roots of
    $P(X)=X^4-aX^2+\frac{1}{4}b^{2}s$ and $P'(X)=X^4-a'X^2+\frac{1}{4}(b')^{2}s'$ respectively.
    Since $\sqrt{s} \in F(\sqrt{s'})$, by Lemma \ref{quad} and the definition of $S$, we have $s=s'$.
    In addition from $[(a+b\sqrt{s})(F(\sqrt{s}))^2]_{S_2}=
    [(a'+b'\sqrt{s'})(F(\sqrt{s}))^2]_{S_2}$, we also have that there exists $r \in \{1,-1\}$ such that
    $\frac{a+b\sqrt{s}}{a'+rb'\sqrt{s}} \in F(\sqrt{s})^2$. This implies that
    $\sigma(\frac{a+b\sqrt{s}}{a'+sb'\sqrt{s}})=
    \frac{a-b\sqrt{s}}{a'-sb'\sqrt{s}} \in F(\sqrt{s})^2$.
    Therefore, it follows from Thereom \ref{elem_class} we have that $F(\sqrt{s})(\beta) \cong F(\sqrt{s})(\beta')$.
    Since by Lemma \ref{unique_elem_closure} we have that the elementary abelian closure of  $F(\beta)$ (resp. $F(\beta'))$
    is $F(\sqrt{s})$ (resp. $F(\sqrt{s'})$), 
   by Theorem \ref{elem_closure_descent}, $\psi$ is indeed well defined.

    Suppose $\psi([(a+b\sqrt{s})(F(\sqrt{s})^\times)^2]_{S_2})=\psi([(a'+b'\sqrt{s'})(F(\sqrt{s'})^\times)^2]_{S_2})$, that
    is $[F(\beta)]_{iso}=[F(\beta')]_{iso}$. Then, by Lemma
    \ref{unique_elem_closure} we have that $E=F(\sqrt{s}) \cong F(\sqrt{s'})$ is the elementary abelian closure of
    $F(\beta)$ and $F(\beta')$.
    By Lemma \ref{elem_closure_descent} we have that $\frac{a^2-b^{2}s}{(a')^2-(b')^{2}s'} \in F^2$
   and $LE/E \cong
    L'E/E$. Hence, by Theorem \ref{elem_class}, 
    $[(a+b\sqrt{s})(F(\sqrt{s})^\times)^2]_{S_2}=[(a'+b'\sqrt{s'})(F(\sqrt{s'})^\times)^2]_{S_2}$.
    So $\psi$ is injective.

    Let $L/F$ be a non-Galois biquadratic generator and $\beta \in L$ be a biquadratic generator with minimal polynomial
    $Q(X)=X^4+uX^2+w$. Since $w \not\in F^2$ by Lemma \ref{unique_elem_closure}, there exists a unique $s \in S$ such that 
    $w=d^{2}s$ for some $d \in F$. We have $-u-2d\sqrt{s} \not\in N_{F(\sqrt{s})}$, since $F(\sqrt{s})$ is the elementary abelian closure of $L/F$ by Lemma \ref{unique_elem_closure}.
    Now $\psi([(-u-2d\sqrt{s})(F(\sqrt{s})^\times)^2]_{S_2})=[F(\beta)]_{iso}$ 
\end{proof}

\begin{definition}\label{Q2}
  Let $F$ be a field and $C$ be a compositum
    of all quadratic fields extensions of $F$.
    If $G$ is a subgroup of $C^\times$ generated by $F$ 
    then relation $\sim_{Q^2}$ on $C^\times \times C^\times$ defined by $(a, a') \sim_{Q^2} (b, b')$ if \begin{enumerate} 
    \item $a=b$ and $a'=b'$ or,
    \item if $(a, a'),(b , b') \in F^\times \times F^\times$, then  $ \frac{a}{b} \in (F^\times)^2$ and  $ \frac{a'}{b'} \in (F^\times)^2$
   \item if $a,b \in F(\delta)^\times -F^\times $ where $\delta^2 \in F^\times$, $\sigma (a) = a'$ and $\sigma (b) = b'$ where $Gal(F(\delta)/F)= {<}\sigma{>}$,
    then $ \frac{a}{b} \in (F(\delta)^\times)^2$.

   \end{enumerate}
   One can prove that $\sim_{Q^2}$ is an equivalence relation. 
\end{definition}

We have now classified all non-cyclic extensions, the following results allows us to see the classes of elementary 
abelian extensions and non-Galois extension as subset of a set.

\begin{theorem} \label{final}
    Let $F$ be a field, $C$ be the compositum of all quadratic field extensions over $F$, then
    $$\{\text{non-cylic biquadratic extension}\}/\sim_{iso}   \simeq  \coprod_{s \in S} \frac{\frac{F(\sqrt{s})^\times -
    N_{F(\sqrt{s})}}{(F(\sqrt{s})^\times)^2}}{S_2}   \coprod \frac{\frac{F^\times}{(F^\times)^2}
\times\frac{F^\times}{(F^\times)^2} -N}{S_3} $$
    where $N$ is a subgroup of $\frac{F^\times}{(F^\times)^2} \times \frac{F^\times}{(F^\times)^2}$ isomorphic to
    $\frac{F^\times}{(F^\times)^2}$,
Moreover, the right hand side can be naturally embedded into $ \frac{\frac{C^\times \times C^\times}{\sim_{Q^2}}}{S_3}$ where $\sim_{Q^2}$ is defined in Definition \ref{Q2} and $S_3$ act on $\frac{C^\times  \times C^\times }{\sim_{Q^2}}$ similarly as defined in
    Lemma \ref{s_3_action}.

\end{theorem}

\begin{proof}
    Let $$\psi_{elem}:\{ L/F \text{ elementary abelian extension of $F$}\} / \sim_{iso} \rightarrow
    \frac{\frac{F^\times}{(F^\times)^2} \times \frac{F^\times}{(F^\times)^2}-N}{S_3}$$ be the bijection defined in Theorem \ref{final_elem_class}
    defined in Theorem \ref{final_elem_class}
    and $$\psi_{non\text{-}G}: \{\text{non-Galois biquadratic extension of $F$}\}/\sim_{iso} \longrightarrow \coprod_{s \in S} \frac{\frac{F(\sqrt{s})^\times - N_{F(\sqrt{s})}}{(F(\sqrt{s})^\times)^2}}{S_2}$$ be the bijection defined in Theorem \ref{final_non_Galois_class}.
    We now define $$\psi: \{\text{non-cylic biquadratic extension of $F$}\}/\sim_{iso}   \longrightarrow  \coprod_{s \in S} \frac{\frac{F(\sqrt{s})^\times -
    N_{F(\sqrt{s})}}{(F(\sqrt{s})^\times)^2}}{S_2}   \coprod \frac{\frac{F^\times}{(F^\times)^2} \times\frac{F^\times}{(F^\times)^2} }{S_3}$$

    as follows, for each $[L/F]_{iso} \in \{\text{non-cylic biquadratic extension of $F$}\}/\sim_{iso}$ we set
    $\psi([L/F]_{iso})=\psi_{elem}([L/F]_{iso})$ if $L/F$ is elementary abelian and $\psi([L/F]_{iso})=\psi_{non-G}([L/F]_{iso})$
    if $L/F$ is non-Galois, it's clear that $\psi$ well-defined is a bijection since both $\psi_{eleme}$ and $\psi_{non-G}$
    are bijections, also the right hand side has an empty intersection.

    We now show $\coprod_{s \in S} \frac{\frac{F(\sqrt{s})^\times - N_{F(\sqrt{s})}}{(F(\sqrt{s})^\times)^2}}{S_2}
    \coprod \frac{\frac{F^\times}{(F^\times)^2} \times\frac{F^\times}{(F^\times)^2} }{S_3}$ can be embedded in
    $\frac{\frac{C^\times \times C^\times}{\sim_{Q^2}}}{S_3}$.

    For each $s \in S$, let $\Delta_{s}: \frac{\frac{F(\sqrt{s})^\times - N_{F(\sqrt{s})}}{(F(\sqrt{s})^\times)^2}}{S_2}
    \longrightarrow \frac{\frac{F(\sqrt{s})^\times}{(F(\sqrt{s})^\times)^2} \times \frac{F(\sqrt{s})^\times}{(F(\sqrt{s})^\times)^2}}{S_3}$
    defined by $\Delta_s([\alpha (F(\sqrt{s})^\times])^2]_{S_2}=
    [\alpha(F(\sqrt{s})^\times)^2,\sigma(\alpha)(F(\sqrt{s})^\times)^2]_{S_3}$ where $\sigma$ is the generator of $Gal(F(\sqrt{s})/F)$.

    To show $\Delta_{s}$ is well-defined, let $[\alpha(F(\sqrt{s})^\times)]_{S_2}=[\beta(F(\sqrt{s})^\times)^2]_{S_2}$,
    then $\alpha=\gamma^2 \beta$ or $\alpha=\gamma^2 \sigma(\beta)$ for some $\gamma \in F(\sqrt{s})^\times$.
    So if $\alpha=\gamma^{2}\beta$ then $\sigma(\alpha)=(\sigma(\gamma))^2 \sigma(\beta) \implies
    \sigma(\alpha)(F(\sqrt{s})^\times)^2=\sigma(\beta)(F(\sqrt{s})^\times)^2$ and
    $\alpha\sigma(\alpha)=(\gamma\sigma(\gamma))^2\beta\sigma(\beta)$ so $\alpha\sigma(\alpha)(F(\sqrt{s})^\times)^2=
    \beta\sigma(\beta)(F(\sqrt{s})^\times)^2$ proving that \\
    $[(\alpha(F(\sqrt{s})^\times)^2,\sigma(\alpha)(F(\sqrt{s})^\times)^2)]_{S_3}=[(\beta(F(\sqrt{s})^\times)^2, \sigma(\beta)(F(\sqrt{s})^\times)^2)]_{S_3}$.
    We can prove the well definedness  similarly when $\alpha=\gamma^2\sigma(\beta)$. 

    As for the injectivity, let $\Delta_{s}([\alpha(F(\sqrt{s})^\times)^2]_{S_2})=\Delta_{s}([\beta(F(\sqrt{s})^\times)^2]_{S_2})$, then
    one of the following statements is true
    \begin{enumerate}
        \item[(a)] $\alpha(F(\sqrt{s})^\times)^2=\beta(F(\sqrt{s})^\times)^2$
        \item[(b)] $\alpha(F(\sqrt{s})^\times)^2=\sigma(\beta)(F(\sqrt{s})^\times)^2$
        \item[(c)] $\alpha(F(\sqrt{s})^\times)^2=\beta\sigma(\beta)(F(\sqrt{s})^\times)^2$
    \end{enumerate}
    If $(a)$ or $(b)$ is true then $[\alpha(F(\sqrt{s})^\times)^2]_{S_2}=[\beta(F(\sqrt{s})^\times)^2]_{S_2}$ and we are done.
  We will now see that $(c)$ cannot be true. Indeed, if it was then $\alpha=\beta\sigma(\beta)\gamma^2$ for some $\gamma \in F(\sqrt{s})^\times$. Moreover, by Lemma \ref{final_non_Galois_class} we have
    $Q(X):=X^4-\frac{1}{2}(\alpha+\sigma(\alpha))X^2+\frac{1}{16}(\alpha-\sigma(\alpha))^2$as a minimal polynomial of some
    biquadratic generator of over $F(\sqrt{s})$, but $(\frac{1}{2}(\alpha+\sigma(\alpha)))^2-
    4(\frac{1}{16}(\alpha-\sigma(\alpha))^2)=\alpha\sigma(\alpha)=(\beta\sigma(\beta) \gamma \sigma (\gamma))^2 \in F(\sqrt{s})^2$
    contradicting the irreducibility of $Q(X)$ over $F(\sqrt{s})$ by Lemma \ref{irred}. 
    So indeed $\Delta_{s}$ is injective.
    
    Those maps induce naturally an embedding $$\Delta : \coprod_{s \in S} \frac{\frac{F(\sqrt{s})^\times - N_{F(\sqrt{s})}}{(F(\sqrt{s})^\times)^2}}{S_2}
    \coprod \frac{\frac{F^\times}{(F^\times)^2} \times\frac{F^\times}{(F^\times)^2} }{S_3} \longrightarrow \coprod_{s \in S} \frac{\frac{F(\sqrt{s})^\times}{(F(\sqrt{s})^\times)^2} \times \frac{F(\sqrt{s})^\times}{(F(\sqrt{s})^\times)^2}}{S_3}
    \coprod \frac{\frac{F^\times}{(F^\times)^2} \times\frac{F^\times}{(F^\times)^2} }{S_3}.$$ 
    
    Finally, by definition of $\sim_{Q^2}$, we obtain a natural embedding $$\iota : \coprod_{s \in S} \frac{\frac{F(\sqrt{s})^\times}{(F(\sqrt{s})^\times)^2} \times \frac{F(\sqrt{s})^\times}{(F(\sqrt{s})^\times)^2}}{S_3}
    \coprod \frac{\frac{F^\times}{(F^\times)^2} \times\frac{F^\times}{(F^\times)^2} }{S_3}
    \hookrightarrow \frac{\frac{C^\times \times C^\times}{\sim_{Q^2}}}{S_3}.$$
    
    The embedding of the theorem is then the composite $\iota \circ \Delta$.

\end{proof}

\end{document}